\documentclass[10pt]{amsart}
\usepackage{amssymb,amscd,euscript}

\numberwithin{equation}{section}

\theoremstyle{plain}
\newtheorem{Theorem}{Theorem}
\newtheorem{Proposition}{Proposition}[section]
\newtheorem{Lemma}[Proposition]{Lemma}
\newtheorem{Corollary}[Proposition]{Corollary}

\theoremstyle{definition}
\newtheorem{Definition}[Proposition]{Definition}

\theoremstyle{remark}
\newtheorem{Remark}[Proposition]{Remark}
\newtheorem*{remark*}{Remark}
\newtheorem{Example}[Proposition]{Example}

\newcommand{\lpar}{{\rm(}\hskip-1pt}
\newcommand{\rpar}{{\rm)}}
\newcommand{\ceps}{\stackrel{\epsilon}{,}}

\renewcommand{\phi}{\varphi}
\renewcommand{\epsilon}{\varepsilon}

\newcommand\p[1]{{{\mathbb P}^#1}}
\newcommand\pp[2]{{{\mathbb P}^#1_#2}}
\newcommand\A[1]{{{\mathbb A}^#1}}
\newcommand\gm{{\mathbf{G_m}}}
\newcommand\gmd{\C[\divisor]^\times}

\newcommand\dif{\mathbf d}

\newcommand\C{{\mathbb C}}
\newcommand\Z{{\mathbb Z}}

\newcommand\Q{{\mathbb Q}}

\newcommand\oM{{\overline M}}
\newcommand\oN{{\overline N}}

\newcommand\cA{{\mathcal A}}
\newcommand\cB{{\mathcal B}}
\newcommand\cC{{\mathcal C}}
\newcommand\cD{{\mathcal D}}
\newcommand\cE{{\mathcal E}}
\newcommand\cF{{\mathcal F}}
\newcommand\cG{{\mathcal G}}

\newcommand\cL{{\mathcal L}}
\newcommand\cM{{\mathcal M}}
\newcommand\cX{{\mathcal X}}
\newcommand\cN{{\mathcal N}}
\newcommand\cP{{\mathcal P}}

\newcommand\cR{{\mathcal R}}
\newcommand\cT{{\mathcal T}}
\newcommand\cS{{\mathcal S}}
\newcommand\cV{{\mathcal V}}
\newcommand\cY{{\mathcal Y}}

\newcommand\fR{{\mathfrak R}}
\newcommand{\divisor}{\mathfrak D}
\newcommand{\Dmod}{\EuScript D}
\newcommand{\iso}{\mathrel{\widetilde\to}}
\newcommand\limind{{\lim\limits_{\longrightarrow}}}
\newcommand\al{{\bullet}}
\newcommand\lb{\vartheta}	

\DeclareMathOperator\gl{\mathfrak{gl}}
\DeclareMathOperator\gr{\mathrm{gr}}
\DeclareMathOperator\supp{\mathrm{supp}}
\DeclareMathOperator\Aut{\mathrm{Aut}}
\DeclareMathOperator\Lie{\mathrm{Lie}}
\DeclareMathOperator\HIGGS{\mathcal{H}\mathit{iggs}}
\DeclareMathOperator\ad{ad}
\DeclareMathOperator\Ker{Ker}
\DeclareMathOperator\Hom{Hom}

\DeclareMathOperator\End{End}
\DeclareMathOperator\END{\mathcal{E}\mathit{nd}}
\DeclareMathOperator\HOM{\mathcal{H}\mathit{om}}
\DeclareMathOperator\TOR{\mathcal{T}\!\mathit{or}}

\DeclareMathOperator\coker{Coker}

\DeclareMathOperator\res{res}
\DeclareMathOperator\spec{Spec}
\DeclareMathOperator\SPEC{\mathcal{S}\mathit{pec}}
\DeclareMathOperator\tr{tr}
\DeclareMathOperator\rk{rk}

\DeclareMathOperator{\cPic}{\overline{\mathcal{P}\mathit{ic}}}
\DeclareMathOperator\Bun{\mathcal{B}\mathit{un}}
\DeclareMathOperator\detrg{detR\Gamma}
\DeclareMathOperator\DR{\mathbb{DR}}
\DeclareMathOperator\HH{{\mathbb H}}
\DeclareMathOperator\id{id}

\DeclareMathOperator\diag{diag}

\newcommand\CPic[1]{\cPic^{\;#1}}  

\newcommand\Mh{{\cM_H}}
\newcommand\Mb{{\overline\cM}}
\newcommand\Nb{{\overline{\mathstrut\cN}}}

\newcounter{noindnum}[subsection]
\newcommand{\stepzero}{\setcounter{noindnum}{0}}
\stepzero

\newcommand{\noindstep}{\noindent\refstepcounter{noindnum}{\rm(}\alph{noindnum}\/{\rm)}\
}

\newcounter{noindnumrm}[subsection]
\newcommand{\noindsteprm}{\noindent\refstepcounter{noindnumrm}{\rm(}\roman{noindnumrm}\/{\rm)}\
}

\begin{document}

\bibliographystyle{alphanum}


\title[An example of the Langlands
correspondence]{An example of the Langlands correspondence for
irregular rank two connections on $\p1$}

\author{D.~Arinkin}
\email{arinkin@email.unc.edu}
\address{Department of mathematics, University of North Carolina,
Chapel Hill, NC}

\author{R.~Fedorov}
\email{fedorov@math.ksu.edu}
\address{Mathematics $\&$ Statistics\\
Boston University\\
111 Cummington St\\
Boston, MA}

\keywords{The Langlands duality; Moduli spaces; Connections
with irregular singularities; Parabolic bundles.}

\begin{abstract}
Special kinds of rank 2 vector bundles with (possibly
irregular) connections on $\p1$ are considered. We construct an
equivalence between the derived category of quasi-coherent
sheaves on the moduli stack of such bundles and the derived
category of modules over a TDO ring on certain non-separated
curve. We identify this curve with the coarse moduli space of
some parabolic bundles on $\p1$. Then our equivalence becomes
an example of the categorical Langlands correspondence.
\end{abstract}

\maketitle

\section{Introduction}

Let ${\mathcal{C}\mathit{onn}}(X,r)$ be the moduli space of
rank~$r$ vector bundles with connections on a smooth complex
projective curve~$X$. Let $\Bun(X,r)$ be the moduli space of
rank~$r$ vector bundles on $X$. The \emph{categorical Langlands
correspondence\/} for $GL(r)$ is a conjectural equivalence
between the derived category of $O$-modules on
${\mathcal{C}\mathit{onn}}(X,r)$ and the derived category of
$\Dmod$-modules on $\Bun(X,r)$. We refer the interested reader
to~\cite[Section 6.2]{FrenkelRecentAdvances}.

This correspondence has been proved by one of the authors in
the settings of rank two bundles equipped with connections with
four simple poles on $X=\p1$ (cf.~\cite{Arinkin}). In this
case, the space $\Bun(X,r)$ should be replaced by the moduli
space of bundles with parabolic structures. More precisely,
\cite{Arinkin} works with $SL(2)$-connections and
$PGL(2)$-bundles. (See~\cite{FrenkelRamifications} for a
discussion of the ramified Langlands program.)

In this paper we extend the results of \cite{Arinkin} to the
case when the ramification divisor still has degree four but we
allow higher order poles as long as leading terms are regular
semisimple (see Theorems~\ref{Langlands} and~\ref{MainTh}).
This provides an example of the categorical Langlands
correspondence for connections with irregular singularities.

In~\cite{FrenkelGross}, Frenkel and Gross present an example of
the Langlands correspondence for a different kind of irregular
singularities. It is instructive to compare the two settings.
Unlike the present paper, the results of Frenkel and Gross
apply to arbitrary group $G$, not just $G=GL(2)$. The
ramification considered is in a sense the simplest nontrivial:
the ramification divisor has degree three. It is proved
in~\cite{FrenkelGross} that in these settings, there is a
unique up to isomorphism local system with prescribed
singularities. In other words, the counterpart of the moduli
space ${\mathcal{C}\mathit{onn}}(X,r)$ consists of a single
point. In particular, the category of $O$-modules on this space
has a unique irreducible object, the structure sheaf of this
point. The corresponding category of automorphic
$\Dmod$-modules (the counterpart of the category of
$\Dmod$-modules on $\Bun(X,r)$) also has a unique irreducible
object \cite[Sections~3,~16]{FrenkelGross}. The categorical
Langlands transform sends the two irreducible generators into
each other.

The present paper studies the `next simplest case': the
ramification divisor has degree four. The moduli space of local
systems is a surface, and the categorical Langlands transform
is an equivalence, similar to the Fourier--Mukai transform.

The techniques used in our argument are similar to that
of~\cite{Arinkin} but more conceptual. We hope that the present
proof is more suitable for generalizations to divisors of
higher order and to the higher genus case.

\begin{remark*}
In positive characteristic, a different approach to Langlands
correspondence was discovered by Bezrukavnikov and Braverman.
In \cite{BravermanBezrukavnikov}, they construct a version of
the categorical Langlands correspondence. In \cite{Nevins},
Nevins uses these ideas for connections with regular
singularities.
\end{remark*}

Our argument uses certain moduli spaces that may be of independent
interest. In Section~\ref{COMPACTIFICATION} we prove that
the moduli space of connections with possibly irregular
singularities has a good moduli space in the sense
of~\cite{Alper}; we also construct a modular projective
compactification of this space, see
Theorems~\ref{GoodModuliSpace} and~\ref{AmpleBundle}. This is
an extension of Simpson's
results~\cite{SimpsonICM,Simpson1,Simpson2}.

Finally, we want to note that our moduli spaces of connections
are the moduli spaces of initial conditions of Painlev\'e
equations. More precisely, the case of regular singularities
corresponds to Painlev\'e~VI, while the cases of irregular
singularities correspond to Painlev\'e~II--Painlev\'e~V,
see~\cite{OhyamaOkumura}.

\subsection{Conventions} We work over the ground field of
complex numbers, thus $\p1$ means $\mathbb{P}^1_\C$, a `scheme'
means a `$\C$-scheme' etc. All schemes and stacks are locally
of finite type.

\subsection{Acknowledgments} We benefited from talks with many
mathematicians. The second author wants to especially thank
David Ben--Zvi, Roman Bez\-ru\-kav\-ni\-kov, Ivan Mirkovic,
Emma Previato, and Matthew Szczesny.

The first author is a Sloan Research Fellow, and he
 is grateful to the Alfred P.~Sloan Foundation for the
support.

\section{Main Results}
Let $\divisor:=\sum n_ix_i$ be a divisor on
$\p1=\mathbb{P}^1_\C$ with $n_i>0$. Let $L$ be a rank 2 vector
bundle on $\p1$, $\nabla:L\to L\otimes\Omega_\p1(\divisor)$ a
connection on $L$ with polar divisor $\divisor$. We call such
pairs $(L,\nabla)$ \emph{connections\/} for brevity.

Choosing a formal coordinate $z$ near $x_i$ and a
trivialization of $L$ on the formal neighborhood of $x_i$, we
can write $\nabla$ near $x_i$ as
\[
\dif+a\frac{\dif z}{z^{n_i}}+\text{higher order terms},\qquad
a\in\gl(2).
\]

In the case $n_i=1$ the connection will be called \emph{non-resonant\/}
if the eigenvalues do not differ by an integer. For $n_i=1$
the conjugacy class of $a$ does not depend on the
choices, so the notion of non-resonant connections does not depend on the choices.

In the case $n_i>1$ the connection will be called
non-resonant at $x_i$ if $a$ has distinct eigenvalues.
For $n_i>1$ the conjugacy class of $a$ does not depend on the
choices up to scaling. Thus
the notion of non-resonant connections is again well defined.

The connection
will be called non-resonant if it is non-resonant at all $x_i$.

\subsection{Moduli stacks}\label{MODST} Let $(L,\nabla)$ be a
non-resonant connection, then in a suitable trivialization of
$L$ over the formal disc at $x_i$ the connection takes a
diagonal form
\begin{equation}\label{fnf}
\nabla=\dif+\begin{pmatrix} \alpha_i^+ & 0\\ 0&\alpha_i^-
\end{pmatrix},
\end{equation}
where $\alpha_i^\pm$ are 1-forms on the formal disc. The polar
parts of these 1-forms do not depend on the trivialization of
$L$, thus we shall call them \emph{the formal type of $\nabla$
at~$x_i$}.

Fix $\divisor$ and for each $i$ polar parts $\alpha_i^\pm$ of
1-forms at $x_i$. Assume that these polar parts satisfy the
following conditions.\\
\noindstep\label{AlphaI} The order of the pole of
$\alpha_i^\pm$ is at most $n_i$, the order of the pole of
$\alpha_i^+-\alpha_i^-$ is exactly
$n_i$.\\
\noindstep\label{AlphaII}
$d:=-\sum_i\res(\alpha_i^++\alpha_i^-)$ is integer.\vskip2pt
\noindstep\label{AlphaIII} $\sum_i\res\alpha_i^\pm\notin\Z$.
(Here for each $i$ there is exactly one summand $\alpha_i^\pm$,
and the choices of signs $+$ and $-$ are independent.)\\
\noindstep\label{AlphaIV} If $n_i=1$, then
$\res\alpha_i^+-\res\alpha_i^-\notin\Z$.

Let $\cM=\cM(\divisor,\alpha_i^\pm)$ be the moduli space of
connections $(L,\nabla)$ such that $\nabla$ has formal types
$\alpha_i^\pm$ at $x_i$. Note that such a connection is
non-resonant by~(\ref{AlphaI}) and~\eqref{AlphaIV}. Also,~$L$
has degree $d$ by~(\ref{AlphaII}).

From now on we assume that $\deg\divisor=4$.

\begin{Theorem}\label{StackProp}
The moduli space $\cM$ is a smooth connected algebraic stack of
dimension $1$. It is a~neutral $\gm$-gerbe over its coarse
moduli space $M$; besides, $M$ is a~smooth quasi-pro\-jec\-tive
surface.
\end{Theorem}
This theorem will be proved in Section~\ref{MODULISPACES}.

Let $\mathrm{Qcoh}(\cM)$ be the category of quasi-coherent
sheaves on $\cM$. Since $\cM$ is a~$\gm$-gerbe, we obtain a
decomposition
\[\mathrm{Qcoh}(\cM)=\prod_{i\in\Z}\mathrm{Qcoh}(\cM)^{(i)},\]
where $\cF\in\mathrm{Qcoh}(\cM)^{(i)}$ if $t\in\gm$ acts on
$\cF$ as $t^i$. Let $\cD^b(\cM)$ be the corresponding bounded
derived category. By definition, objects of $\cD^b(\cM)$ are
complexes of $O_\cM$-modules with quasi-coherent cohomology. It
follows from \cite[Claim~2.7]{ArinkinBezrukavnikov} that
$\cD^b(\cM)$ is equivalent to the bounded derived category of
$\mathrm{Qcoh}(\cM)$. Thus we also have a decomposition
\[
    \cD^b(\cM)=\prod_{i\in\Z}\cD^b(\cM)^{(i)}.
\]
It is easy to see that $\cF\in\cD^b(\cM)^{(i)}$ if and only if
$H^\bullet(\cF)\in\mathrm{Qcoh}(\cM)^{(i)}$.

\subsection{Twisted differential operators}\label{TDO}
Denote by $\wp:P\to\p1$ the projective line with points $x_i$
doubled. In other words, $P$ is obtained by gluing two copies
of~$\p1$ outside the support of $\divisor$. Denote the
preimages of $x_i$ by $x_i^-$ and $x_i^+$. Let
$j:\p1-\divisor\hookrightarrow P$ be the natural embedding.
This notation will be used throughout the paper.

The main result of the present paper is that
$\cD^b(\cM)^{(-1)}$ is equivalent to a category of twisted
$\Dmod$-modules on $P$. To give a precise definition of this
twist, recall that the isomorphism classes of sheaves of rings
of twisted differential operators (TDO) on a smooth (not
necessarily separated) curve are classified by the first
cohomology group of the sheaf of 1-forms.

\begin{Lemma}\label{CohP}
Denote by $\omega_i$ the vector space of polar parts of 1-forms
at $x_i\in\p1$. Then
\[
  H^1(P,\Omega_P)=\C\oplus\bigoplus_i\omega_i.
\]
\end{Lemma}
\begin{proof}
Let $D_i^\pm$ be the formal disc at $x_i^\pm$. These discs,
together with $\p1-\divisor$, give a~cover of $P$; let us use
the corresponding \v{C}ech complex. We see that a~1-cocycle is
a~collection $\beta_i^\pm$ of 1-forms on punctured formal
discs, and one easily checks that the map
$(\beta_i^\pm)\mapsto(\,\sum_i\res(\beta_i^++\beta_i^-),\beta_i^+-\beta_i^-)$
induces the required isomorphism.
\end{proof}

Using this lemma, we define the sheaf of differential operators
on $P$ twisted by $(-d,\alpha_i^+-\alpha_i^-)$; denote it by
$\Dmod_{P,\alpha}$. In other words, it is given by the
1-cocycle $(\alpha_i^\pm)$.

\subsection{The integral transform}
Let $\xi=(L,\nabla)\in\cM$. Denote by $\xi_\alpha$ the
$\Dmod_{P,\alpha}$-module generated by $\wp^*\xi$.
More precisely, since the twist of $\Dmod_{P,\alpha}$
is supported outside of $\p1-\divisor$, we can and shall identify
the restrictions of $\Dmod_{P,\alpha}$ and $\Dmod_P$ to $\p1-\divisor\subset P$. Thus we
have the middle extension functor $j_{!*}$ from the category of $\Dmod_{\p1-\divisor}$-modules to
the category of $\Dmod_{P,\alpha}$-modules. We set
$\xi_\alpha:=j_{!*}(\xi|_{\p1-\divisor})$.
\begin{Remark}\label{ForDisc}
Let us describe the restriction of $\xi_\alpha$ to
$\wp^{-1}(D_i)$, where $D_i$ is the formal disc centered
at~$x_i$. Choose 1-forms $\tilde\alpha^\pm$ with polar parts
$\alpha_i^\pm$. According to~\eqref{fnf}, the restriction
of~$\xi$ to~$D_i$ is isomorphic to
\[
    (O_{D_i},\dif+\tilde\alpha^-)\oplus(O_{D_i},\dif+\tilde\alpha^+).
\]
Now one checks easily that
\[
\begin{split}
\xi_\alpha|_{D_i^-}&\simeq
(O_{D_i^-},\dif+\tilde\alpha^-)\oplus(O_{\dot D},\dif+\tilde\alpha^+),\\
\xi_\alpha|_{D_i^+}&\simeq (O_{\dot
D},\dif+\tilde\alpha^-)\oplus(O_{D_i^+},\dif+\tilde\alpha^+),
\end{split}
\]
where $\dot D\subset D_i^\pm$ is the punctured formal disc.
\end{Remark}

Note that formal normal form exists for families of
connections, and it is constant for families in $\cM$. Thus our
middle extension construction still makes sense for families of
connections in $\cM$. Hence we can apply it to the universal
family $\xi$ on $\cM\times\p1$, getting an $\cM$-family
$\xi_\alpha$ of $\Dmod_{P,\alpha}$-modules. In other words,
$\xi_\alpha$ is an $O_\cM\boxtimes\Dmod_{P,\alpha}$-module on
$\cM\times P$. Thus $\xi_\alpha$ gives rise to an integral
transform from $\cD^b(\cM)$ to the derived category of
$\Dmod_{P,\alpha}$-modules.

Denote the natural projections of $\cM\times P$ to $\cM$ and
$P$ by $p_1$ and $p_2$ respectively.

\begin{Theorem}\label{MainTh}
Let $d$ be an odd number. Then the functor
\[
\Phi_{\cM\to P}:\cF\mapsto
Rp_{2,*}\Bigl(\xi_{\alpha}\mathop{\otimes}\limits_{O_{\cM\times
P}}p_1^*\cF\Bigr)
\]
provides an equivalence between $\cD^b(\cM)^{(-1)}$ and the bounded
derived category of $\Dmod_{P,\alpha}$-modules.
\end{Theorem}

Theorem~\ref{MainTh} is the main result of the paper; we prove it
in Sections~\ref{PRELYM}--\ref{LYSENKO}.

\begin{Remark}\label{RemOdd}
\noindstep It is easy to see that the restriction of
$\Phi_{\cM\to P}$ to $\cD^b(\cM)^{(i)}$ is zero unless $i=-1$.

\noindstep\label{Remodd2} On the other hand, $\cD^b(\cM)^{(i)}$
depends only on parity of $i$. Indeed, fix $x\in\p1-\divisor$,
and let $\delta_x$ be the line bundle on $\cM$ whose  fiber at
$(L,\nabla)$ is equal to $\det L_x$. Then the tensor product
with $\delta_x$ provides an equivalence between
$\cD^b(\cM)^{(i)}$ and $\cD^b(\cM)^{(i+2)}$.

\noindstep\label{RemOdd3} Assume that $\divisor=\sum n_ix_i$ is
not even, that is one of the numbers $n_i$ is odd, then all the
categories $\cD^b(\cM)^{(i)}$ are equivalent. Indeed, let $n_i$
be odd, and for $(L,\nabla)\in\cM$ let $\eta_i$ be a unique
level $n_i$ parabolic structure at $x_i$ compatible with
$\nabla$ (see~Definition~\ref{ParB} and Section~\ref{AFFSTR}).
Tensoring with the line bundle whose fiber at $(L,\nabla)$ is
$\det\eta_i$, we get an equivalence between the odd and the
even components of the derived category.

\noindstep In fact our theorem is also valid if $d$ is even but
$\divisor$ is not even. Indeed, let $n_i$ be odd and define a
collection~$\beta_i^\pm$ of polar parts of 1-forms by
\[
    \beta_i^+=\alpha_i^++n_i\,\frac{\dif z}z,\qquad
    \beta_i^-=\alpha_i^-,\qquad
    \beta_j^\pm=\alpha_j^\pm\text{ for }i\ne j.
\]
Then a modification at $x_i$ provides an isomorphism
$\cM(\alpha)\simeq\cM(\beta)$, and we can apply the theorem to
$\cM(\beta)$. (See Section~\ref{PARBUN} for the definition of
modification.) It remains to notice that the category of
$\Dmod_{P,\alpha}$-modules is equivalent to the category of
$\Dmod_{P,\beta}$-modules: the equivalence is given by
tensoring with a line bundle.
\end{Remark}

\subsection{The Langlands Correspondence}\label{LANGLANDS}
\begin{Definition}\label{ParB}
Let $L$ be a rank 2 vector bundle on $\p1$. A
\emph{level-$\divisor$ parabolic structure} on $L$ is a line
subbundle $\eta$ in the restriction of $L$ to $\divisor$ (we
view $\divisor$ as a non-reduced subscheme of $\p1$). We call a
bundle with a parabolic structure \emph{a~parabolic bundle}.
\end{Definition}

Let
$\overline\Bun(d^\vee)=\overline\Bun(\p1,2,d^\vee,\divisor)$ be
the moduli stack of rank 2 degree $d^\vee$ vector bundles on
$\p1$ with level-$\divisor$ parabolic structures. (We reserve
notation $\Bun$ for its open substack of bundles without
non-scalar endomorphisms, cf. Section~\ref{AFFSTR}.)

Let $\C[\divisor]$ be the ring of functions on the scheme
$\divisor$, $\gmd$ be the group of invertible functions; this
is an algebraic group. Choosing local coordinates at the
points~$x_i$, we get an isomorphism
\[
\gmd=\prod_i(\C[z]/z^{n_i})^\times.
\]
Let $\pi:\eta_{univ}\to\overline\Bun(d^\vee)$ be the
$\gmd$-torsor whose fiber over
$(L,\eta)\in\overline\Bun(d^\vee)$ is
\[
\{s\in H^0(\divisor,\eta)|\;s(x_i)\ne0\;\text{for all }i\}.
\]

The collection $\alpha_i^+$ of polar parts of 1-forms can be
viewed as an element of
\[
    (\Lie(\gmd))^\vee=\C[\divisor]^\vee
\]
via the residue pairing. Thus it gives rise to a TDO ring on
$\overline\Bun(d^\vee)$ through non-commutative reduction of
the sheaf of differential operators on the total space
of~$\eta_{univ}$ (see Section~\ref{TDOEXISTENCE}). We denote this TDO
ring by $\Dmod_{\overline\Bun(d^\vee),\alpha^+}$.

Similarly, we define a $\gmd$-torsor $\eta'_{univ}$ whose
fiber over $(L,\eta)\in\overline\Bun(d^\vee)$ is
\[
\{s\in H^0(\divisor,(L|_\divisor)/\eta)|\;s(x_i)\ne0\;
\text{ for all }i\}.
\]
Denote the TDO ring corresponding to $\eta'_{univ}$ and the
collection $\alpha_i^-$ by
$\Dmod_{\overline\Bun(d^\vee),\alpha^-}$. Let
$\Dmod_{\overline\Bun(d^\vee),\alpha}$ be the Baer sum of the
TDO rings $\Dmod_{\overline\Bun(d^\vee),\alpha^+}$ and
$\Dmod_{\overline\Bun(d^\vee),\alpha^-}$.

\begin{Theorem}[The Langlands Correspondence]\label{Langlands}
Assume that $d$ is an odd number. Then $\cD^b(\cM)^{(-1)}$ is
equivalent to the bounded derived category of
$\Dmod_{\overline\Bun(-1),\alpha}$-modules.
\end{Theorem}

Theorem~\ref{Langlands} is derived from Theorem~\ref{MainTh} in Section~\ref{ProofOfLanglands}.

\begin{Remark}\label{RemEven}
\noindstep Let us discuss the notion of the derived category of
$\Dmod_{\overline\Bun(-1),\alpha}$-modules. Note that
$\overline\Bun(-1)$ is a (smooth) algebraic stack, so this
notion is not immediate.

As we show in Section~\ref{ProofOfLanglands} the category of
$\Dmod_{\overline\Bun(-1),\alpha}$-modules is equivalent to a
category of twisted $\Dmod$-modules on $P$ (in fact $P$ is the
coarse moduli space of a certain open subset of
$\overline\Bun(-1)$). Thus we shall view the derived category
of twisted $\Dmod$-modules on $P$ as the definition for the
derived category of $\Dmod_{\overline\Bun(-1),\alpha}$-modules.
(See also the discussion in Section~\ref{ProofOfLanglands}.)

\noindstep In general, we expect an equivalence of categories
between $\cD^b(\cM)^{(d^\vee)}$ and the bounded derived
category of $\Dmod_{\overline\Bun(d^\vee),\alpha}$-modules.
This statement follows from Theorem~\ref{Langlands} if $d^\vee$
is odd. Indeed, pick $x\in\p1-\divisor$, then
$(L,\eta)\mapsto(L(x),\eta)$ is an isomorphism between
$\overline\Bun(d^\vee)$ and $\overline\Bun(d^\vee+2)$. It
remains to use Remark~\ref{RemOdd}(\ref{Remodd2}).

\noindstep We also have the desired equivalence if $\divisor$
is not even. Indeed, let $n_i$ be odd. Modification at $x_i$
gives an isomorphism
$\Dmod_{\overline\Bun(d^\vee),\alpha}\simeq
\Dmod_{\overline\Bun(d^\vee+n_i),\beta}$, where $\beta$ is
obtained from $\alpha$ by swapping $\alpha_i^+$ and
$\alpha_i^-$. It remains to use the previous
remark, Remark~\ref{RemOdd}(\ref{RemOdd3}), and the obvious
identification $\cM(\alpha)=\cM(\beta)$.
\end{Remark}

\subsection{The plan of proof of
Theorem~\ref{MainTh}}\label{PRELYM} Theorem~\ref{MainTh}
reduces to two orthogonality statements.

Let $p_{12}:P\times P\times\cM\to P\times P$ and
$p_{13},p_{23}:P\times P\times\cM\to P\times\cM$ be the
projections. Let $\xi^\vee$ be the vector bundle on
$\p1\times\cM$ dual to $\xi$. Since it has a connection along
$\p1$, we see that $\xi^\vee_\alpha:=(\xi^\vee)_{-\alpha}$ is a
$\Dmod_{P,-\alpha}\boxtimes O_\cM$-module on $P\times\cM$.

Set
$\cF_P:=(p_{13}^*\xi_\alpha)\otimes(p_{23}^*\xi_\alpha^\vee)$.
Here $p_{13}^*$ and $p_{23}^*$ stand for the $O$-module
pullback (from the viewpoint of $\Dmod$-modules, these pullback
functors should include a cohomological shift). Note that
$\xi_\alpha$ is a flat $O_{P\times\cM}$-module (see
Remark~\ref{ForDisc}), hence
\[
    (p_{13}^*\xi_\alpha)\otimes(p_{23}^*\xi_\alpha^\vee)=
    (p_{13}^*\xi_\alpha)\otimes^L(p_{23}^*\xi_\alpha^\vee).
\]

Further, $Rp_{12,*}\cF_P$ is an object of the derived category
of $p_1^\al\Dmod_{P,\alpha}\circledast
p_2^\al\Dmod_{P,-\alpha}$-modules, where $p_1,p_2:P\times P\to
P$ are the projections. Here $p_i^\al$ (resp.\ $\circledast$)
stands for the inverse image (resp.\ Baer sum) of TDO rings
(the corresponding functors on Lie algebroids are described in
\cite{BeilinsonBernstein}).

\begin{Theorem}\label{Theorem3}
$Rp_{12,*}\cF_P=\delta_\Delta[-1]$, where $\Delta\subset
P\times P$ is the diagonal, and $\delta_\Delta$ is the direct
image of $O_{\Delta}$ as a $\Dmod_\Delta$-module.
\end{Theorem}
This theorem is proved in Section~\ref{ORTHOGONALITY}.

\begin{Remark}
In general, for a map $f:X\to Y$ and a TDO ring $\Dmod_1$ on
$Y$, there is a functor
$f_+:\cD^b(f^\al\Dmod_1)\to\cD^b(\Dmod_1)$, where
$\cD^b(\Dmod_1)$ is the bounded derived category of
$\Dmod_1$-modules. For the embedding $i:\Delta\hookrightarrow
P\times P$, one easily checks that
$i^\al(p_1^\al\Dmod_{P,\alpha}\circledast
p_2^\al\Dmod_{P,-\alpha})$ is the (non-twisted) differential
operator ring $\Dmod_\Delta$, so $\delta_\Delta:=i_+(O_\Delta)$
is well defined as a $p_1^\al\Dmod_{P,\alpha}\circledast
p_2^\al\Dmod_{P,-\alpha}$-module.
\end{Remark}

By Theorem~\ref{Theorem3}, $\xi_\alpha$ is an orthogonal
$P$-family of vector bundles on $\cM$. To obtain an equivalence
of categories, one should also show that $\xi_\alpha$ is
orthogonal as an $\cM$-family of $\Dmod_{P,\alpha}$-modules.
Let us give the precise statement.

Consider $\cF_\cM:=p_{13}^*\xi_\alpha\otimes
p_{23}^*\xi_\alpha^\vee$ (here
$p_{13},p_{23}:\cM\times\cM\times P\to \cM\times P$ are the
projections).

$\cF_\cM$ can be viewed as a family of $\Dmod_P$-modules
parameterized by $\cM\times\cM$. Consider the de Rham complex
of $\cF_\cM$ in the direction of $P$
\[
  \DR(\cF_\cM)=\DR_P(\cF_\cM):=(\cF_\cM\to\cF_\cM\otimes
  \Omega_{\cM\times\cM\times P/\cM\times\cM}).
\]
Our aim is to compute $Rp_{12,*}\DR(\cF_\cM)$.

By Theorem~\ref{StackProp}, $\cM\times\cM$ is a
$\gm\times\gm$-gerbe over a scheme, so $\gm\times\gm$ acts on
any quasi-coherent sheaf $\cF$ on $\cM$. Therefore, $\cF$ can
be decomposed with respect to the characters of $\gm\times\gm$.
Denote by $\cF^\psi$ the component of $\cF$ corresponding to
the character $\psi:\gm\times\gm\to\gm$ defined by
$(t_1,t_2)\mapsto t_1/t_2$.

Let $\diag:\cM\to\cM\times\cM$ be the diagonal morphism.
\begin{Theorem}[Lysenko]\label{ThLys}
  $Rp_{12,*}\DR(\cF_\cM)=(\diag_*O_\cM)^\psi[-2]$.
\end{Theorem}

In \cite{LysRS}, Lysenko calculates the scalar product of automorphic sheaves corresponding to
irreducible rank $n$ local systems by developing a geometric analogue of the Rankin-Selberg method.
Theorem~\ref{ThLys} is essentially a similar calculation. Unfortunately, Theorem~\ref{ThLys} would
require the Rankin-Selberg method for ramified local systems; this setting is not considered in \cite{LysRS}.
However, as we work with a very special kind of local systems, Theorem~\ref{ThLys} is easy to prove directly.
We prove it in Section~\ref{LYSENKO}; the argument is parallel to Lysenko's calculation from unpublished notes
\cite{Lysenko}.

\begin{Remark}\label{RemDiag}
Note that $\cM$ is
a $\gm$-torsor over $\diag(\cM)$. Thus
\[
    \diag_*O_\cM=\bigoplus_{i\in\Z}(\diag_*O_\cM)^{\psi^i}.
\]

Objects of $\cD^b(\cM\times\cM)$ define endofunctors on $\cD^b(\cM)$. Let us consider the functors
corresponding to the components of $\diag_*O_\cM$. Clearly,
\[\cD^b(\cM)\to\cD^b(\cM):\cF\mapsto Rp_{2,*}((\diag_*O_\cM)\otimes p_1^*\cF)\]
is isomorphic to the identity functor. It is easy to see that the functor
\[\cD^b(\cM)\to\cD^b(\cM):\cF\mapsto Rp_{2,*}((\diag_*O_\cM)^{\psi^i}\otimes p_1^*\cF)\]
is isomorphic to the projection $\cD^b(\cM)\to\cD^b(\cM)^{(-i)}$.
\end{Remark}

\begin{proof}[Proof of Theorem~\ref{MainTh}]
The inverse to the functor $\Phi_{\cM\to P}$ is given by
\[
\Phi_{P\to\cM}:\cF\mapsto
Rp_{1,*}\DR(\xi_\alpha^\vee
\otimes p_2^*\cF)[2].
\]
Indeed, using base change and Theorem~\ref{Theorem3}, one
checks that the composition $\Phi_{\cM\to
P}\circ\Phi_{P\to\cM}$ is isomorphic to the identity functor.
Similarly, it follows from Theorem~\ref{ThLys} and
Remark~\ref{RemDiag} that the composition
$\Phi_{P\to\cM}\circ\Phi_{\cM\to P}$ is isomorphic to the
projection $\cD^b(\cM)\to\cD^b(\cM)^{(-1)}$.
\end{proof}

\section{A compactification of moduli spaces of
connections}\label{COMPACTIFICATION} In this section, we
compactify a moduli space of connections
with singularities following Simpson (\cite{SimpsonICM,Simpson1,
Simpson2}).

In~\cite{Simpson2}, Simpson constructs a natural
compactification of the moduli space of vector bundles with
connections on a smooth projective variety $X$. We consider the
case when $X$ is a smooth projective curve. In this case it is
not hard to generalize the result to the case of connections
with singularities (we use~\cite{Simpson1}). Then we prove that
the compactification is in fact projective (note that for
varieties of higher dimension projectivity of the
compactification is not known). Our description of the divisor
at infinity is also more explicit than in~\cite{Simpson2}.

The compactification is constructed in the following
generality. Let $X$ be a smooth complex projective curve, $r$ a
positive integer, $d$ an integer, and $\divisor$ an effective
divisor on $X$. Denote by $\cN=\cN(X,r,d,\divisor)$ the moduli
stack of pairs $(L,\nabla)$, where $L$ is a vector bundle on
$X$ of rank $r$ and degree $d$, and $\nabla:L\to
L\otimes\Omega_X(\divisor)$ is a connection on $L$ with the
order of poles bounded by $\divisor$. Our goal is to construct
a compactification of the semistable part of $\cN$.

Fix $X$, $r$, $d$, and $\divisor$.

\subsection{$\epsilon$-connections}
We shall construct the compactification as a moduli space of
Deligne's $\lambda$-connections. Recall the following
\begin{Definition}\label{EpsilonConnections}
Let $L$ be a vector bundle on $X$. For a one-dimensional vector
space $E$ and $\epsilon\in E$, an
\emph{$\epsilon$-connection\/} on $L$ is a $\C$-linear map
$\nabla:L\to L\otimes\Omega_X\otimes_\C E$ such that
\[
 \nabla(fs)=f\nabla s+s\otimes\dif f\otimes\epsilon\;\text{
 for all }f\in O_X,s\in L.
\]
More generally, an $\epsilon$-connection on $L$ with poles
bounded by $\divisor$ is a map $\nabla:L\to
L\otimes\Omega_X(\divisor)\otimes_\C E$ satisfying the same
condition.
\end{Definition}

Denote by $\Nb=\Nb(X,r,d,\divisor)$ the moduli stack of
collections $(L,\nabla;\epsilon\in E)$, where $L$ is a vector
bundle on $X$ of rank $r$ and degree $d$, and $\nabla$ is an
$\epsilon$-connection on $L$ with poles bounded by $\divisor$.
This is an algebraic stack, the proof is similar
to~\cite[Proposition~1]{FedorovIsoStokes}.

\begin{Definition}
$(L,\nabla;\epsilon\in E)\in\Nb$ is \emph{semistable\/} if for
any non-zero $\nabla$-invariant subbundle $L_0\subset L$ we
have
\[
  \frac{\deg L_0}{\rk L_0}\le\frac{d}{r}.
\]
Further, $(L,\nabla;\epsilon\in E)$ is \emph{nilpotent} if
$\epsilon=0$ and $\nabla^r=0$. Note that if $\epsilon=0$,
$\nabla$ is $O_X$-linear, so $\nabla^r$ makes sense as a map
$L\to L\otimes(\Omega_X(\divisor)\otimes_\C E)^{\otimes r}$.
Equivalently, $(L,\nabla;\epsilon\in E)$ is nilpotent if there
is a flag of subbundles
\[
    0=L_0\subset L_1\subset\dots\subset L_k=L
\]
with $\nabla(L_i)\subset
L_{i-1}\otimes\Omega_X(\divisor)\otimes_\C E$.
\end{Definition}

Let $\Nb^{ss}\subset\Nb$ be the open substack of semistable
$\epsilon$-connections.

Also, let $\Nb^{\,ss,nn}\subset\Nb^{ss}$ be the open substack
of semistable $\epsilon$-connections that are not nilpotent.

Taking $E=\C$, $\epsilon=1$, we see that connections are
particular cases of $\epsilon$-connections. Moreover if
$\epsilon\ne0$, there is a unique isomorphism $E\to\C$ such
that $\epsilon\mapsto 1$. It follows that the open substack of
$\Nb$ corresponding to $\epsilon$-connections with
$\epsilon\ne0$ parameterizes all connections $(L,\nabla)$,
where $L$ has rank $r$ and degree $d$, $\nabla$ has poles
bounded by $\divisor$. Thus this substack can be identified
with $\cN$.

We use the theory of good moduli spaces developed by J.~Alper
(see \cite{Alper}). By definition, a quasi-compact map
$p:\cS\to S$ from a stack $\cS$ to an algebraic space~$S$ is a
\emph{good moduli space} if the direct image functor $p_*$ is
exact on quasi-coherent sheaves and $p_*O_\cS=O_S$. In
particular, this notion reduces to the notion of quotient in
the sense of geometric invariant theory when $\cS$ is the
quotient stack of a scheme by an action of an algebraic group
(see \cite[Theorem~13.6]{Alper}). Note that
by~\cite[Theorem~4.16(vi)]{Alper} $p$ is universal among maps
to schemes.

\begin{Theorem}[cf.\ Theorem~11.3 of~\cite{Simpson2}]
\label{GoodModuliSpace}\noindstep There is a good moduli space
\lpar{}in the sense of~\cite{Alper}\rpar{}
$p:\Nb^{\,ss,nn}\to\oN$ such that $\oN$ is a complete scheme.

\noindstep\label{GoodModB} Set
$\cN^{ss}=\Nb^{ss}\cap\cN=\Nb^{\,ss,nn}\cap\cN$, then
$\cN^{ss}$ has a good moduli space $N$, which is an open
subscheme of $\oN$ \lpar{}so $N$ is the moduli space of
semistable bundles with connections\rpar.
\end{Theorem}

\begin{Remark}\label{SecondFromFirst}
Part~\eqref{GoodModB} of the theorem is clear from the
construction of $\oN$ in Section~\ref{CONSTRUCTION} but it also
easily follows from the first claim and
Corollary~\ref{BundlesOnM}. Indeed, let $\cE$ be the line
bundle on $\Nb^{\,ss,nn}$ whose fiber over
$(L,\nabla;\epsilon\in E)$ is $E$. As shown in
Corollary~\ref{BundlesOnM}, there is a line bundle $\cE'$ on
$\oN$ such that $p^*\cE'=\cE^{\otimes r!}$. Then
by~\cite[Proposition~4.5]{Alper} we can identify
$p_*\cE^{\otimes r!}$ with $\cE'$. Thus $\epsilon^{\otimes r!}$
gives a section of $\cE'$; let~$N$ be the complement of its
zero locus. Clearly $N$ is open and $\cN=p^{-1}(N)$. This
implies the claim.
\end{Remark}

\subsection{Construction of $\oN$} \label{CONSTRUCTION}
The moduli of $\epsilon$-connections (on arbitrary projective
variety) is constructed by Simpson (\cite{SimpsonICM,Simpson1,
Simpson2}); it is easy to see that his argument works in the
case of $\epsilon$-connections with singularities on curves.
Let us quickly recall the construction.

Fix $\epsilon\in\C$, and denote by $\cN_\epsilon$ the moduli
stack of $\epsilon$-connections of the form
$(L,\nabla;\epsilon\in\C)$. Equivalently, $\cN_\epsilon$ is a
fiber of the map
\[
  \Nb\to\A1/\gm:(L,\nabla;\epsilon\in E)\mapsto(\epsilon\in
E).
\]
Here we identify the quotient stack $\A1/\gm$ with the moduli
stack of pairs $(\epsilon\in E)$, where $E$ is a
one-dimensional vector space.

As $\epsilon$ varies, the stacks $\cN_\epsilon$ form a family
$\cN_\al\to\A1$ (whose fiber over $\epsilon\in\A1$
is~$\cN_\epsilon$). The total space $\cN_\al$ carries an action
of $\gm$ via
\[
  t\cdot(L,\nabla;\epsilon\in\C)=(L,t\nabla;t\epsilon\in\C),\quad
  (L,\nabla;\epsilon\in\C)\in\cN_\al,\;t\in\gm.
\]
We can identify $\Nb$ with the quotient stack $\cN_\al/\gm$.

Denote by $\cN_\epsilon^{ss}\subset\cN_\epsilon$
(resp.~$\cN^{ss}_\al\subset\cN_\al$) the open substacks of
semistable $\epsilon$-connections.

\begin{Proposition}
\noindstep\label{GmsI} There exists a good moduli space
$\cN_\epsilon^{ss}\to N_\epsilon$;

\noindstep\label{GmsII} As $\epsilon\in\C$ varies, the spaces
$N_\epsilon$ form a family $N_\al\to\A1$ whose fiber over
$\epsilon\in\A1$ is $N_\epsilon$. There exists a good moduli
space $\cN^{ss}_\al\to N_\al$.
\end{Proposition}
\begin{proof}
(\ref{GmsI}) This is a particular case
of~\cite[Theorem~4.10]{Simpson1}. Note
that~\cite[Theorem~4.10]{Simpson1} applies to the moduli space
of semistable modules over a sheaf of split almost polynomial
rings of differential operators $\Lambda$ (see
\cite[Section~2]{Simpson1}, p.~77, 81 for definition). In our case,
$\Lambda$ is the universal enveloping of the Lie algebroid
\[
    \Lambda^{\le1}=(O_X\oplus\cT_X(-\divisor),
    [\cdot\ceps\cdot],\rho),
\]
where $\rho$ is the composition of the natural inclusion
$\cT_X(-\divisor)$ into $\cT_X$ with multiplication by
$\epsilon$, and
\[
[f_1+\tau_1\ceps f_2+\tau_2]=
\epsilon(\tau_1(f_2)-\tau_2(f_1)+[\tau_1,\tau_2]).
\]

(\ref{GmsII}) Consider the sheaf $p_1^*\Lambda$ on
$X\times\A1$, where $\Lambda$ is the sheaf from~(\ref{GmsI})
with $\epsilon=1$. Let $\Lambda^R$ be its subsheaf generated by
the operators of the form $\sum\epsilon^i\lambda_i$, where
$\lambda_i\in\Lambda$ has order at most $i$, $\epsilon$ is the
coordinate on $\A1$. The family $N^\al$ is constructed by
applying Theorem~4.10 to $\Lambda^R$ relative to the projection
$X\times\A1\to\A1$. See \cite[Section on $\tau$-connections,
p.~87]{Simpson1}.
\end{proof}

The action of $\gm$ on $\cN_\al$ induces its action on $N_\al$.
(In fact an action of an algebraic group on a stack always
induces an action on the good moduli space; for the proof, use
the universal property of good moduli spaces
and~\cite[Proposition~4.7(i)]{Alper}.) In particular, a point
$z\in N_\al$ yields a morphism $\gm\to N_\al:t\mapsto t\cdot
z$. If it can be extended to a morphism from $\A1\supset\gm$
(resp.\ from $\p1-\{0\}\supset\gm$), we say that the limit
$\lim_{t\to0}t\cdot z$ (resp.~$\lim_{t\to\infty}t\cdot z$)
exists.

The Hitchin fibration gives the following description of $N_0$.
Let
\[
    B:=\prod_{i=1}^r H^0(X,\Omega_X(\divisor)^{\otimes i})
\]
be the base of Hitchin fibration. Recall that the Hitchin map
sends a Higgs bundle $(L,\nabla;0\in\C)$ to
$(c_1(\nabla),\dots,c_r(\nabla))$, where $c_i(\nabla)$ are
coefficients of the characteristic polynomial of $\nabla$. Thus
we get a map $\cN_0\to B$. This map descends to a map
\[
    (c_1,c_2,\ldots,c_r):N_0\to B.
\]
Denote by $N_\al^n\subset N_0$ the zero fiber of Hitchin
fibration.

\begin{Lemma}
Let $z\in N_\al$ correspond to
$(L,\nabla;\epsilon\in\C)\in\cN^{ss}_\al$. Then
$\lim_{t\to\infty}t\cdot z$ exists if and only if
$(L,\nabla;\epsilon\in\C)$ is nilpotent.
\end{Lemma}
\begin{proof}
Let us start with the `only if' direction. If the limit exists,
then $\lim_{t\to\infty}t\epsilon$ exists, so $\epsilon=0$.
Also, the coefficients of the characteristic polynomial of
$t\nabla$ are equal to $t^i c_i(\nabla)$, and so the limit
$\lim_{t\to\infty}t^ic_i(\nabla)$ exists. Therefore,
$c_i(\nabla)=0$; in other words, $\nabla$ is nilpotent.

To prove the `if' direction, it suffices to notice that
$N_\al^n$ is complete. Indeed, it is the zero fiber of the
Hitchin map. This map is proper, the proof is similar to
(\cite{Simpson1}, Theorem~6.11). More precisely, one has to
repeat the proof of that theorem, changing $T^*$ to
$T^*(\divisor)$.
\end{proof}

\begin{Proposition}
The geometric quotient $(N_\al-N_\al^n)/\gm$ exists; the
quotient is a complete scheme of finite type, the natural
projection $N_\al-N_\al^n\to(N_\al-N_\al^n)/\gm$ is an affine
map.
\end{Proposition}

\begin{proof}
This follows from~\cite[Theorem~11.2]{Simpson2} and the
previous lemma. Indeed, the fixed point set is closed in
$N_\al^n$ and thus complete. The fact that $\lim_{t\to0}t\cdot
z$ exists for all $z\in N_\al$ follows
from~\cite[Theorem~10.1]{Simpson2} (see also Corollary~10.2).

To show that the map is affine, note
that~\cite[Theorem~11.2]{Simpson2} is derived from
~\cite[Theorem~11.1]{Simpson2}, which in turn uses
Proposition~1.9 in~\cite{MumfordGIT}, but this proposition also
claims that the map is affine.
\end{proof}

The geometric quotient $\oN=(N_\al-N_\al^n)/\gm$ has the
properties required in Theorem~\ref{GoodModuliSpace} because of
the following

\begin{Lemma}
Let $p:\cS\to S$ be a good moduli space. Let $G$ be a reductive
group acting on $\cS$. Consider the induced action on $S$ and
assume that there is a geometric quotient $S/\!/G$ such that
the projection $S\to S/\!/G$ is affine. Then the induced map
$\bar p:\cS/G\to S/\!/G$ is a good moduli space \lpar{}the
quotient in the left hand side is the stacky one\rpar.
\end{Lemma}
\begin{proof}
Let us decompose $\bar p$ as
\[
\cS/G\xrightarrow{p'}S/G\xrightarrow{p''}S/\!/G.
\]
We just have to check that $p'_*$ and $p''_*$ are exact and
take the structure sheaves to the structure sheaves. This
easily follows from our assumptions.
\end{proof}

\subsection{Projectivity of $\oN$}
Let us construct an ample bundle on $\oN$. Fix a point $x\in
X$.

\begin{Lemma} \label{TrivialAction}
Let $\alpha$ be an automorphism of $(L,\nabla;\epsilon\in
E)\in\Nb^{\,ss,nn}$.\\
\noindstep\label{TrivAct1} The action of $\alpha$ on $E$ is a
root of unity of degree at most $r$. In particular,
$\alpha$~acts trivially on $E^{\otimes r!}$.\\
\noindstep\label{TrivAct2} If $\alpha$ acts trivially on $E$,
then $\alpha$ acts trivially on
\[
    \detrg(X,L)^{\otimes r}\otimes\det(L_x)^{\otimes(rg-d-r)}.
\]
\noindstep\label{TrivAct3} The automorphism $\alpha$ acts
trivially on
\[
    \left(\detrg(X,L)^{\otimes r}\otimes\det(L_x)^{\otimes
    (rg-d-r)}\right)^{\otimes r!}.
\]
\end{Lemma}
\begin{proof}
(\ref{TrivAct1}) Note that $\alpha(\epsilon)=\epsilon$, so if
$\epsilon\ne0$, then $\alpha$ acts trivially on $E$. If
$\epsilon=0$, consider the coefficients of the characteristic
polynomial of $\nabla$
\[c_i\in H^0(X,\Omega_X(\divisor)^{\otimes i})\otimes
E^{\otimes i}.\] Since $\nabla$ is not nilpotent, there is $i$
such that $c_i\ne 0$. Now it suffices to note that
$\alpha(c_i)=c_i$.\\
(\ref{TrivAct2}) We can decompose $L=\oplus_{\lambda\in\C}
L^\lambda$, where $\alpha-\lambda$ is nilpotent on $L^\lambda$
(almost all of the summands vanish). Since $\alpha$ acts
trivially on $E$, we have $\nabla(L^\lambda)\subset
L^\lambda\otimes\Omega_X(\divisor)\otimes_\C E$. By
semistability of $L$, $\deg L^\lambda=\frac{d}{r}\rk
L^\lambda$. We can then identify
\[
    \detrg(X,L)\simeq\bigotimes\detrg(X,L^\lambda)
\]
and $\det L_x\simeq\bigotimes\det((L^\lambda)_x)$. Finally,
$\alpha$ acts as $\lambda^{\deg L^\lambda-(g-1)\rk L^\lambda}$
(here $g$ is the genus of~$X$) on $\detrg(X,L^\lambda)$ and as
$\lambda^{\rk L_\lambda}$ on
$\det((L^\lambda)_x)$.\\
(\ref{TrivAct3}) Follows from (\ref{TrivAct2}) applied to
$\alpha^{r!}$.
\end{proof}

Let us denote by $\delta$ (resp.\ $\cE$, resp.\ $\delta_x$) the
line bundle on $\Nb$ whose fiber over $(L,\nabla;\epsilon\in
E)$ equals $\detrg(X,L)$ (resp.\ $E$, resp.\ $\det(L_x)$).

\begin{Corollary}\label{BundlesOnM}
The line bundles
\[
 \cE^{\otimes r!}|_{\Nb^{\,ss,nn}}\quad\text{and}\quad
 \left.\left(\delta^{\otimes
 r}\otimes\delta_x^{\otimes(rg-d-r)}\right)^{\otimes r!}\right|_{\Nb^{\,ss,nn}}
\]
are pullbacks of line bundles on $\oN$.
\end{Corollary}
\begin{proof}
By Lemma~\ref{TrivialAction}, automorphisms of any closed point
of $\Nb^{\,ss,nn}$ act trivially on the fibers of these two
bundles. The statement now follows from
\cite[Theorem~10.3]{Alper}.
\end{proof}

Denote the corresponding line bundles on $\oN$ by $\cE'$ and
$\delta'$.

\begin{Theorem}\label{AmpleBundle}
The line bundle $(\delta')^{-1}\otimes(\cE')^{\otimes k}$ is
ample on $\oN$ for $k\gg 0$.
\end{Theorem}

\begin{proof}
Recall that by construction, $(\delta')^{-1}$ is ample on $N$,
cf.\ Remark before~\cite[Theorem 4.10]{Simpson1}.
Let $N_H$ be $\oN-N$ with reduced scheme
structure. By construction, $N_H$ is the reduction of the
quotient scheme $(N_0-N_\al^n)/\gm$. The Hitchin map therefore
induces a morphism from $N_H$ to the quotient scheme
$(B-\{0\})/\gm$, which is a weighted projective space.
Recall~\cite[Definition~4.6.1]{EGAII} that a sheaf $\cF$ on $Y$
is called relatively ample for a map $f:Y\to Y'$ if for some
cover $Y'=\cup Y'_\alpha$ with affine $Y'_\alpha$, the sheaves
$\cF|_{f^{-1}(Y'_\alpha)}$ are ample.

We make the following observations.\stepzero

\indent\noindstep $(\delta')^{-1}$ is relatively ample for the
Hitchin map $N_0\to B$ (since $(\delta')^{-1}$ is ample on
$N_0$);

\indent\noindstep $(\delta')^{-1}$ is relatively ample for the
equivariant Hitchin map $N_H\to(B-\{0\})/\gm$. Indeed, the
relative ampleness can be proved fiberwise,
cf.~\cite[Theorem~4.7.1]{EGAIII} (We are thankful to Brian
Conrad for the reference). On the other hand, a fiber of the
equivariant Hitchin map is the categorical quotient of the
corresponding fiber of the Hitchin map by the finite stabilizer
of the corresponding point in $(B-\{0\})/\gm$. It remains to
use the following general fact: \emph{the descent of an
equivariant ample line bundle to the quotient by a finite group
is ample}. This follows from~\cite[Proposition~1.15,
Theorem~1.10]{MumfordGIT};

\indent\noindstep $\cE'|_{N_H}$ is naturally a pullback of a
sheaf on $(B-\{0\})/\gm$, which we also denote by $\cE'$;

\indent\noindstep $\cE'$ is very ample on $(B-\{0\})/\gm$.

Recall from Remark~\ref{SecondFromFirst} that
$\epsilon^{\otimes r!}$ yields a section $\epsilon'\in
H^0(\oN,\cE')$, whose set-theoretic zero locus is $N_H$. Denote
by $N_H'$ the scheme-theoretic zero locus of $\epsilon'$. It is
a non-reduced `thickening' of $N_H$.

\emph{Step 1.} For integers $l,k$, consider the line bundle
\[
 \cL=\cL_{l,k}:=(\delta')^{\otimes-l}\otimes\cE'^{\otimes
k}.
\]
There exists $l_0$ such that for all $l>l_0$, there is
$k_0=k_0(l)$ such that for all $k>k_0$ the line bundle
$\cL|_{N_H}$ is very ample on $N_H$. This follows
from~\cite[Proposition~4.6.11]{EGAII}
and~\cite[Proposition~4.4.10(ii)]{EGAII}.

\emph{Step 2.} For any coherent sheaf $\cF$ on $N_H$, there
exists $l_0$ such that for all $l>l_0$, there is $k_0=k_0(l)$
such that for all $k>k_0$ and all $i>0$,
\[
  H^i(N_H,\cF\otimes\cL)=0
\]
and $\cF\otimes\cL$ is generated by global sections. This
follows from the fact that the derived functor of global
sections on $N_H$ is the composition of the derived direct
image to $(B-\{0\})/\gm$ and the derived functor of global
sections on $(B-\{0\})/\gm$.

\emph{Step $2'$.} Same statement as in Step 2 is true with
$N_H$ changed to $N'_H$. For the proof, consider a filtration
of $\cF$ with factors supported scheme-theoretically on $N_H$
and use the long exact sequence of cohomology.

\emph{Step $1'$.} Same statement as in Step 1 is true with
$N_H$ changed to $N'_H$. Indeed, set $\cF_i:=O_{N'_H}(-iN_H)$.
By Step $2'$ we can assume that $\cF_1\otimes\cL$ is generated
by global sections and $H^0(N'_H,\cL)$ surjects onto
$H^0(N_H,\cL)$. By Step~1 we can assume that $\cL|_{N_H}$ is
very ample. Let us show that $\cL$ is very ample on $N_H'$.

We are going to use~\cite[Proposition~II.7.2]{Hartshorne}. Take
$s\in H^0(N'_H,\cL)$ and let~$N'_s$ be the open subset of
$N'_H$ defined by $s\ne0$. Then $N_s:=N'_s\cap N_H$ is affine.
Therefore,~$N'_s$ is also affine. It suffices to show that the
set $\{s'/s|\,s'\in H^0(N'_H,\cL)\}$ generates the ring
$A_i:=H^0(N'_s,O_{N'_H}/\cF_i)$ for all $i$. We proceed by
induction. For $i=1$ this follows from very ampleness
of~$\cL|_{N_H}$. Take $t\in A_i$; using our statement with
$i=1$, we can assume that $t\in\cF_1/\cF_i$. By assumption it
can be written as $\sum\lambda_j(s_j/s)$, where $s_j\in
H^0(N'_H,\cF_1\otimes\cL)$, $\lambda_j\in A_{i-1}$. It remains to use the
inductive hypothesis.

For $i\gg0$ we have $\cF_i=0$ and we are done.

\emph{Step $3$.} From now on, fix $l$ satisfying the conditions
of Steps $1'$ and $2'$, and also such that
$\cL|_N=(\delta')^{\otimes-l}|_N$ is very ample. Then the
restriction map $H^0(\oN,\cL)\to H^0(N'_H,\cL)$ is surjective
for $k\gg0$.

Indeed, this map fits into the exact sequence
\[
  H^0(\oN,\cL)\to H^0(N'_H,\cL)\to
  H^1(\oN,\cL(-N'_H))\to H^1(\oN,\cL)\to H^1(N'_H,\cL)
\]
According to Step $2'$, for $k\gg0$, the rightmost term
vanishes, and so the map $H^1(\oN,\cL(-N'_H))\to H^1(\oN,\cL)$
is surjective. On the other hand, $\cL(-N'_H)$ equals
$\cL_{l,k-1}$, so $\dim H^1(\oN,\cL)$ decreases as a function
of $k$ for $k\gg0$. This dimension is finite by properness of
$\oN$, and therefore stabilizes. In other words, for $k\gg 0$,
the map $H^1(\oN,\cL(-N'_H))\to H^1(\oN,\cL)$ is an
isomorphism.

\emph{Step $4$.} $\cL=\cL_{l,k}$ is very ample on $\oN$ for
$k\gg0$ (the choice of $l$ is as in Step $3$).

Recall that $\mathbb{P}(V)$ denotes the projective space of
hyperplanes in a vector space $V$. Choose a finite-dimensional
vector space $V\subset H^0(N,\cL)$ that defines an embedding
$N\hookrightarrow\mathbb{P}(V)$, and for every $k\gg0$ a
subspace $W_k\subset H^0(N'_H,\cL)$ that defines an embedding
$N'_H\hookrightarrow\mathbb{P}(W_k)$. For $k\gg 0$, the space
$V$ is contained in $H^0(\oN,\cL\otimes(\cE')^{-1})\subset
H^0(N,\cL)$ (because $H^0(N,\cL)$ is the limit of spaces
$H^0(\oN,\cL)$ as $k\to\infty$), and $W_k$ can be lifted to a
finite-dimensional subspace of $H^0(\oN,\cL)$ (which we still
denote by $W_k$) by Step $3$. It follows
from~\cite[Proposition~II.7.3]{Hartshorne} that
$V\epsilon'+W_k+W_{k-1}\epsilon'$ defines an embedding
$\oN\hookrightarrow{\mathbb{P}}(V\epsilon'+W_k+W_{k-1}\epsilon')$
for $k\gg0$. Note that this
proposition is stated for projective schemes only but it is
valid for any proper scheme. Indeed, the projectivity is needed
only for applying \cite[Corollary~5.20]{Hartshorne}, but the
corollary is well-known to be true with the weaker assumption.
\end{proof}

\section{Properties of $\cM$ and of its
compactification}\label{MODULISPACES}

In this section we prove Theorem~\ref{StackProp}. We also prove
\begin{Proposition}\label{CohomologicalDimension}
Let $\cF$ be any quasi-coherent sheaf on $\cM$. Then
$H^i(\cM,\cF)=0$ for $i\ge2$.
\end{Proposition}

We construct a compactification $\Mb=\cM\sqcup\Mh$ (see
Proposition~\ref{Cartier}). We prove that the stacks $\Mb$ form
a flat family as $\divisor$ and the local invariants of
connections vary (Proposition~\ref{PrFlat}). We give an
explicit description of parabolic bundles underlying bundles
with connections (Proposition~\ref{BunP}). We begin with
general statements but starting from Lemma~\ref{Irreducible} we
assume that $X=\p1$ and $\deg\divisor=4$. This assumption
continues through the end of the paper.

\subsection{Connections compatible with parabolic structure}
We start by describing parabolic bundles that possess
compatible connections.

Let $X$ be a smooth projective curve of genus $g$.  Let $L$ be
a vector bundle of degree $d$ on $X$. Denote by $b(L)\in
H^1(X,\END(L)\otimes\Omega_X)$ the class of the Atiyah sequence
\[
    0\to L\otimes\Omega_X\to B(L)\to L\to 0.
\]
We have the Serre duality pairing
$\langle\cdot,b(L)\rangle:\End(L)\to\C$. Recall from
\cite[Proposition~18]{Atiyah} that
\begin{equation}\label{atiyah}
\begin{split}
\langle A,b(L)\rangle &=0\text{ if $A$ is nilpotent},\\
\langle\id_L,b(L)\rangle &=-d.
\end{split}
\end{equation}

\begin{Remark}
\noindstep To match \cite{Atiyah}, the Serre duality pairing
should include the factor of $2\pi\sqrt{-1}$.

\noindstep For every $A\in \End(L)$ the Serre duality pairing
$\langle A,b(L)\rangle$ is given by
\[
    \langle A,b(L)\rangle =-\tr(A|R\Gamma(X,L))+\tr(A)\chi(O_X).
\]
Here $\tr(A | R\Gamma(X,L))$ is the alternating sum of the
traces of maps on $H^i(X,L)$ induced by $A$, and
$\chi(O_X)=1-g$. This follows from~\eqref{atiyah}. Indeed, we
can decompose $L$ into a direct sum such that the semisimple
part of $A$ is scalar on each summand.
\end{Remark}

Fix $\epsilon\in\C$, distinct points $x_1,\dots,x_k\in X$, and
principal parts
\[
    A_i\in\END(L)(\infty\cdot x_i)/\END(L)
\]
for the vector bundle $\END(L)$ at $x_i$ for all $i=1,\dots,k$.
To these data, we associate the sheaf of $\epsilon$-connections
$\cC=\cC(L,\epsilon,\{A_i\}_{i=1}^k)$: its sections over an
open subset $U\subset X$ are $\epsilon$-connections
\[
    \nabla:L|_U\to L|_U\otimes\Omega_U\left(\sum_{i=1}^k\infty\cdot x_i\right)
\]
such that $\nabla-A_i$ is regular at $x_i$ for all $x_i\in U$.
Since $\cC$ is a torsor over $\END(L)\otimes\Omega_X$, its
isomorphism class is given by an element
\[
    c=c\,(L,\epsilon,\{A_i\}_{i=1}^k)\in H^1(X,\END(L)\otimes\Omega_X).
\]

\begin{Lemma}\label{lm:PolarParts} For every $A\in\End(L)$,
\[
    \langle A,c\rangle=\epsilon\langle A,b(L)\rangle+
    \sum_{i=1}^k\tr(\res(A\cdot A_i)).
\]
\end{Lemma}
\begin{proof} The torsor $\cC$ depends linearly on the
collection $(L,\epsilon,\{A_i\}_{i=1}^k)$. We can therefore
assume that either all $A_i=0$ (and then $c=\epsilon b(L)$) or
$\epsilon=0$ (this case follows from the definition of Serre
pairing).
\end{proof}

It is easy to adapt Lemma~\ref{lm:PolarParts} to the settings
of parabolic bundles. For simplicity, we only state it for
bundles of rank two. Let us fix a divisor $\divisor=\sum
n_ix_i\ge 0$ on~$X$. Suppose that $L$ is a rank two vector
bundle on $X$, and $\eta$ is a level $\divisor$ parabolic
structure on $L$, that is, a line subbundle $\eta\subset
L|_\divisor$ (cf.\ Definition~\ref{ParB}).

Denote by $\END(L,\eta)$ the locally free sheaf of
endomorphisms of $L$, preserving $\eta$; let
$\End(L,\eta):=H^0(\p1,\END(L,\eta))$ be the corresponding ring
of endomorphisms.

\begin{Corollary}\label{ConnExist}
Fix $\epsilon\in\C$ and principal parts
$\alpha^+,\alpha^-\in\Omega_X(\divisor)/\Omega_X$. The
following conditions are equivalent:

\stepzero\noindstep\label{cond:parabolic} There exists an
$\epsilon$-connection $\nabla:L\to L\otimes\Omega_X(\divisor)$
whose `polar part'
$L|_\divisor\to(L\otimes\Omega_X(\divisor))|_\divisor$ equals
to $\alpha^+$ on $\eta$ and induces $\alpha^-$ on
$(L|_\divisor)/\eta$.

\noindstep For any endomorphism $A\in\End(L,\eta)$, we have
\[
    \res(A_+\alpha_+)+\res(A_-\alpha_-)+\epsilon
    \langle A, b(L)\rangle=0.
\]
Here $A_+, A_-\in \C[\divisor]$ are the \lpar{}scalar\rpar{}
operators induced by $A$ on $\eta$ and on $(L|_\divisor)/\eta$
respectively, and the residue functional
$\res:\Omega_X(\divisor)/\Omega_X\to\C$ is given by
\[\res\omega:=\sum_{x\in\divisor}\res_x\omega.\]
\end{Corollary}
\begin{proof}
Denote by $\HIGGS(L,\eta)$ the sheaf of Higgs fields $B:L\to
L\otimes\Omega_X(\divisor)$ whose polar part
$L|_\divisor\to(L\otimes\Omega_X(\divisor))|_\divisor$ induces
$0$ on both $\eta$ and $(L|_\divisor)/\eta$. In other words, in
any $\eta$-compatible local trivialization of $L|_{n_ix_i}$ we
have
\[
B=\begin{pmatrix}
0 & *\\
0 & 0
\end{pmatrix}
+\text{non-singular terms}.
\]
Note that
$\HIGGS(L,\eta)\simeq\END(L,\eta)^\vee\otimes\Omega_X$. Indeed,
the pairing is given by the trace of the product, and one
checks in local coordinates that it is perfect.

The sheaf of connections satisfying \eqref{cond:parabolic}
forms a torsor over $\HIGGS(L,\eta)$; clearly, this torsor is
induced from the $\END(L)\otimes\Omega_X$-torsor
$\cC(L,\epsilon,\{A_i\}_{i=1}^n)$ for a choice of polar parts
$A_i$ compatible with $\eta$ and $\alpha^\pm$. Now the claim
follows from Lemma~\ref{lm:PolarParts}.
\end{proof}

\begin{Remark}
If $A_i$ has a pole of first order for all $i$, then
Corollary~\ref{ConnExist} becomes Theorem~7.2 in
\cite{Crawley-Boevey}, which is a special case of Mihai's
results \cite{Mihai1,Mihai2}.
\end{Remark}

\subsection{Local invariants of connections,
revisited}\label{DEFMB} Let $\cN=\cN(X,2,d,\divisor)$ be as
in Section~\ref{COMPACTIFICATION}. Let $(L,\nabla;\epsilon\in
E)\in\Nb$ be an $\epsilon$-connection, where $L$ has rank 2.
Let $D$ be the formal disc centered at $x_i$. Trivializing
$L|_D$, we can write
\[
\nabla|_D=\epsilon\dif+a,\qquad a\in\gl(2)\otimes
\Omega_D(n_ix_i)\otimes_\C E.
\]
It is easy to see that $\tr a$ and $\det a$ are well defined
(that is independent of the trivialization) as sections of
$E\otimes_\C(\Omega_X(n_ix_i)/\Omega_X)$ and
$E^{\otimes2}\otimes_\C(\Omega^{\otimes2}_X(2n_ix_i)/
\Omega^{\otimes2}_X(n_ix_i))$ respectively. Performing this
operation at every $x_i$, we get well-defined sections of
$E\otimes_\C(\Omega_X(\divisor)/\Omega_X)$ and
$E^{\otimes2}\otimes_\C
(\Omega^{\otimes2}_X(2\divisor)/\Omega^{\otimes2}_X(\divisor))$,
which we denote $[\tr\nabla]$ and $[\det\nabla]$ respectively.

Clearly, in the case of a non-resonant connection
$(L,\nabla;1\in\C)$ we get
\[
[\tr\nabla]=\nu_1:=(\alpha_i^++\alpha_i^-),\qquad
[\det\nabla]=\nu_2:=(\alpha_i^+\alpha_i^-),
\]
where $(\alpha_i^\pm)$ is the formal type of the connection
(cf. Section~\ref{MODST}). Thus in this case the data
$\bigl([\tr\nabla],[\det\nabla]\bigr)$ is equivalent to the
formal type.

Fix $\nu_1\in\Omega_X(\divisor)/\Omega_X$ and
$\nu_2\in\Omega^{\otimes2}_X(2\divisor)/
\Omega^{\otimes2}_X(\divisor)$ and denote by $\Mb$ the closed
substack of $\Nb^{\,ss,nn}$ parameterizing
$\epsilon$-connections such that
\begin{equation}\label{hitchincond}
[\tr\nabla]=\epsilon\otimes\nu_1,\qquad
[\det\nabla]=\epsilon^{\otimes2}\otimes\nu_2.
\end{equation}

Assume now that $X=\p1$, $\deg\divisor=4$, $d$ is odd. Recall
that in Section~\ref{MODST} we defined a moduli stack $\cM$.

\begin{Lemma}\label{Irreducible}
Every connection $(L,\nabla)\in\cM$ is irreducible.
\end{Lemma}
\begin{proof}
Assume for a contradiction that $L'\subset L$ is a
$\nabla$-invariant line subbundle. One checks that
$\res_{x_i}(\nabla|_{L'})=\res\alpha_i^\pm$, where
$\alpha_i^\pm$ are defined in Section~\ref{MODST}. This contradicts
condition~(\ref{AlphaIII}) of Section~\ref{MODST}
(cf.~\cite[Proposition~1]{ArinkinLysenko}).
\end{proof}

In particular every $(L,\nabla)\in\cM$ is semistable, and we
see that the open substack of~$\Mb$ given by $\epsilon\ne0$ is
identified with $\cM$. Let $\Mh$ be the closed substack
of~$\Mb$ defined by $\epsilon=0$. Then $\cM=\Mb-\Mh$.
By~Theorem~\ref{GoodModuliSpace} and~\cite[Lemma~4.14]{Alper},
$\Mb$ has a good moduli space $\oM$. It follows from
Theorem~\ref{AmpleBundle} that $\oM$ is projective.

Note that $\cM$ is a closed substack of $\cN^{ss}$, so, using
again~\cite[Lemma~4.14]{Alper}, we see that $M:=N\cap\oM$ is
the good moduli space of $\cM$. Clearly, $M$ is
open in $\oM$.

\begin{Proposition}\label{Cartier}
$\Mh\subset\Mb$ is a Cartier divisor.
\end{Proposition}

This will be proved below after we prove some properties of
$\cM$.

\subsection{An affine bundle structure on $\cM$}\label{AFFSTR}
Denote by $\Bun(d)=\Bun(\p1,2,d,\divisor)$ the moduli stack of
level-$\divisor$ parabolic bundles $(L,\eta)$ such that $L$ has
degree $d$ and $\End(L,\eta)=\C$. By semicontinuity $\Bun(d)$
is an open substack in $\overline\Bun(d)$.

Consider $(L,\nabla)\in\cM$. The formal
classification~(\ref{fnf}) of connections shows that there is a
unique level-$\divisor$ parabolic structure $\eta$ compatible
with $\nabla$ in the following sense: for any $i$ and every
section $s$ of $L$ in a neighborhood of $x_i$ such that
$s|_{n_ix_i}\in H^0((n_ix_i),\eta)$ we have that $\nabla
s-\alpha_i^+s$ is regular at $x_i$.

\begin{Proposition}\label{Affine}
\noindstep\label{Aff0} If $(L,\nabla)\in\cM$, then $L\simeq
O_\p1(m)\oplus
O_\p1(n)$ with $m-n=1$.\\
\noindstep\label{Aff1} Assume that $(L,\eta)$ is the parabolic
bundle corresponding to $(L,\nabla)\in\cM$, then
$\End(L,\eta)=\C$.\\
\noindstep\label{Aff2} The resulting map $\rho:\cM\to\Bun(d)$
is an affine bundle of rank 1.
\end{Proposition}
\begin{proof}
(\ref{Aff0}) Let us write $L=O_\p1(m)\oplus O_\p1(n)$ with
$m>n$. Assume for a contradiction that $m-n\ge3$. Consider the
map
\[
    \overline\nabla:O_\p1(m)\hookrightarrow L\xrightarrow{\nabla}
    L\otimes\Omega_\p1(\divisor)
    \twoheadrightarrow O_\p1(n)\otimes\Omega_\p1(\divisor)\simeq O_\p1(n+2).
\]
It is easy to see that this map is $O_\p1$-linear, thus it is
zero. It follows that $O_\p1(m)$ is a $\nabla$-invariant
subbundle in $L$, which contradicts Lemma~\ref{Irreducible}.

(\ref{Aff1}) Assume first that there is $A\in\End(L,\eta)$ such
that $A$ has different eigenvalues. Then in exactly the same
way as in~\cite[Proposition~3]{ArinkinLysenko} we come to
contradiction with condition (\ref{AlphaIII}) of Section~\ref{MODST}.

Assume now that $A\in\End(L,\eta)$ is not scalar and it has
equal eigenvalues. Then the matrix of $A$ with respect to the
decomposition $L=O_\p1(m)\oplus O_\p1(n)$ is
\[
A=
\begin{pmatrix}
c & f\\
0 & c
\end{pmatrix},
\]
where $c$ is a constant, $f$ is a section of $O_\p1(m-n)$, that
is a polynomial of degree at most 1.

For every $i$ choose the maximal $m_i$ such that
$\eta|_{m_ix_i}=O_\p1(m)|_{m_ix_i}$. It is easy to see that $A$
preserves $\eta$ if and only if $f$ vanishes at $x_i$ at least
to order $n_i-2m_i$ for all $i$. Hence the existence of
non-scalar endomorphism implies
\begin{equation}\label{ineq}
\sum_i n_i-2m_i\le1.
\end{equation}

Consider again $\overline\nabla:O_\p1(m)\to O_\p1(n+2)$. It
follows from the compatibility that~$\overline\nabla$ has zero
of order at least $m_i$ at $x_i$. Again, $\overline\nabla\ne0$,
so $\sum m_i\le1$. However this inequality together
with~(\ref{ineq}) would imply $\deg\divisor\le3$.

(\ref{Aff2}) Consider $(L,\eta)\in\Bun(d)$. Combining
part~\eqref{Aff1}, the second formula in~\eqref{atiyah},
condition~\eqref{AlphaII} of Section~\ref{MODST}, and
Corollary~\ref{ConnExist}, we see that the fiber of $\rho$ over
$(L,\eta)$ is non-empty. Thus it is a torsor over
$H^0(\p1,\HIGGS(L,\eta))$. Using the identification
$\HIGGS(L,\eta)=\END(L,\eta)^\vee\otimes\Omega_X$, we obtain
\begin{multline}\label{dimcalc}
\dim H^0(\p1,\HIGGS(L,\eta))=\dim H^1(\p1,\END(L,\eta))=1-\chi(\END(L,\eta))=\\
1-\deg(\END(L,\eta))-\rk(\END(L,\eta))=1.
\end{multline}
We see that $H^0(\p1,\HIGGS(L,\eta))$ form a vector bundle over
$\Bun(d)$, and $\cM$ is a torsor over this bundle.
\end{proof}

\begin{Remark}
Note that, contrary to the case of regular singularities, this
proposition is not valid for $n>4$ because the proof of
part~(\ref{Aff1}) is specific for this case.
\end{Remark}

\subsection{Parabolic bundles}\label{PARBUN}
Let $(L,\eta)$ be a rank 2 parabolic bundle. We define the
\emph{lower modification} of $(L,\eta)$ at $x_i$ as the vector
bundle $L_i$ whose sheaf of sections is
\[
\{s\in L:\,s|_{n_ix_i}\in\eta|_{n_ix_i}\}.
\]

Clearly, $\deg L_i=\deg L-n_i$. We shall also use the lower
modification $L_\eta$ of $L$ at all $x_i$: its sheaf of
sections is
\[
\{s\in L:\,s|_\divisor\in\eta\}.
\]
Note that $\eta$ induces parabolic structures on $L_i$ and
$L_\eta$. For example in the case of~$L_\eta$ we get an exact
sequence
\[
0\to\eta\otimes O_\p1(-\divisor)|_\divisor\to
L(-\divisor)|_\divisor\to L_\eta|_\divisor,
\]
thus the image $\eta'$ of $L(-\divisor)|_\divisor$ in
$L_\eta|_\divisor$ is a parabolic structure on $L_\eta$. Upon
choosing local coordinates $z_i$ at each $x_i$, $\eta'$ can be
identified with $(L|_\divisor)/\eta$. It is easy to see that
$\End(L,\eta)=\End(L_\eta,\eta')$ and a similar statement is
true for $L_i$ with induced parabolic structure.

Recall that $P$ is a projective line doubled at the points of
the support of $\divisor$.

\begin{Proposition}\label{BunP}
$\Bun(d)\simeq P\times B(\gm)$, where $B(\gm):=pt/\gm$ is the
classifying stack of $\gm$.
\end{Proposition}
\begin{proof}
The idea of the proof was suggested to the first author by
Drinfeld.

\emph{Step 1. We can assume $d=-1$.} Let us pick a point
$\infty\in\p1-\divisor$. Then the map $(L,\eta)\mapsto(L\otimes
O_\p1(\infty),\eta)$ identifies $\Bun(\p1,2,d,\divisor)$ with
$\Bun(\p1,2,d+2,\divisor)$. Since $d$ is odd, the statement
follows.

\emph{Step 2. $L\simeq O_\p1\oplus O_\p1(-1)$.} Indeed, it
follows from Proposition~\ref{Affine}(\ref{Aff2}) that
$(L,\eta)$ corresponds to a connection $(L,\nabla)\in\cM$, thus
we can use Proposition~\ref{Affine}(\ref{Aff0}).

\emph{Step 3.} The discussion, preceding this proposition,
shows that $(L,\eta)\in\Bun(-1)$ implies $(L_\eta\otimes
O_\p1(2\infty),\eta')\in\Bun(-1)$ and therefore $L_\eta\simeq
O_\p1(-2)\oplus O_\p1(-3)$.

Let $\tilde P$ be the moduli stack of collections $(L,\eta,
O_\p1\hookrightarrow L, O_\p1(-2)\hookrightarrow L_\eta)$,
where $(L,\eta)\in\Bun(-1)$. Note that there is a unique up to
scalar map $O_\p1\to L$ and a unique up to scalar map
$O_\p1(-2)\to L_\eta$. Thus $\tilde P$ is a principal
$\gm\times\gm$-bundle on $\Bun(-1)$.

\emph{Step 4.} For a point of $\tilde P$ we get a map
$\phi:O_\p1\oplus O_\p1(-2)\to L$. We claim that this map is
injective. Indeed, let $m_i$ be as in the proof of
Proposition~\ref{Affine}(\ref{Aff1}). If the image of $\phi$ is
a line subbundle, then $\sum m_i\ge2$. But we saw (again in the
proof of Proposition~\ref{Affine}(\ref{Aff1})) that this is
impossible.

Thus $\phi$ has a simple zero at a single point $q$. Note that
$\Ker\phi(q)$ does not coincide with the fiber of $O_\p1$
(because $O_\p1\to L$ is an embedding of vector bundles). That
is, the kernel of $\phi(q)$ is spanned by $(p,1)$, where $p$ is
a point in the fiber of $O_\p1(2)$ over $q$. (More canonically,
$p$ is a homomorphism from the fiber of $O_\p1(-2)$ to that of
$O_\p1$.) The pair $(p,q)$ completely describes $L$ as an upper
modification of $O_\p1\oplus O_\p1(-2)$: the sheaf of sections
of $L(-q)$ is
\begin{equation}\label{upper1}
    \{(s_1,s_2)\in O_\p1\oplus
O_\p1(-2)|\,s_1(q)=ps_2(q)\}.
\end{equation}

\emph{Step 5}. Similarly, we get a map
$\phi':O_\p1(-\divisor)\oplus O_\p1(-2)\to L_\eta$. It also has
exactly one simple zero. Note that $\det\phi=\det\phi'$ (since
$\phi$ and $\phi'$ can be identified on $\p1-\divisor$), so the
zero is at the same point $q$. Then $\Ker\phi'(q)$ is spanned
by $(1,p')$, where $p'$ is in the fiber of $O_\p1(2)$ (more
properly, $p'$ is a homomorphism between the fiber of
$O_\p1(-\divisor)$ and that of $O_\p1(-2)$). Again, $(p',q)$
completely determines~$L_\eta$, indeed, the sheaf of sections
of $L_\eta(-q)$ is
\begin{equation}\label{upper2}
    \{(s_1,s_2)\in O_\p1(-\divisor)\oplus
    O_\p1(-2)|\,p's_1(q)=s_2(q)\}.
\end{equation}

\emph{Step 6.} Note that $(p,p',q)$ determines the inclusion
$L_\eta\hookrightarrow L$ uniquely as well because it
determines it on $\p1-\divisor$, thus this triple determines a
point of $\tilde P$. We must have $L_\eta\subset L$. Looking
at~(\ref{upper1}) and~(\ref{upper2}) it is easy to see that
this condition is exactly $pp'=f(q)$, where $f$ is the
canonical section of $O_\p1(\divisor)$ (thus the zero locus of
$f$ is exactly $\divisor$). This makes sense: the product $pp'$
is in the fiber of $O_\p1(\divisor)$.

Let $P'$ be the set of triples $(p,p',q)$ as above subject to
the condition $pp'=f(q)$. Every such point determines a
parabolic bundle $(L,\eta)$ but some of these bundles can have
extra automorphisms. In other words, $\tilde P\subset P'$.

Clearly, $P'$ is fibered over $\p1$ with coordinate $q$, and
the fiber over $x$ is either a hyperbola or a cross, depending
on whether $x$ is in $\divisor$ or not.

Finally, we need to mod out the embeddings
$O_\p1\hookrightarrow L$ and $O_\p1(-2)\hookrightarrow L_\eta$.
If we scale one of them by $a$ and the other by $b$, we get
\[
    (p,p',q)\mapsto((a/b)p,(b/a)p',q).
\]

Therefore
\begin{itemize}
\item The only points with extra automorphisms are of the
    form $(0,0,q)$, $q\in\divisor$ (the centers of the
    crosses);

\item The stable locus $\tilde P$ is exactly the part of
    $P'$ that is smooth over $\p1$.

\item We have $\Bun(-1)=\tilde P/\mathbf{G_m^\mathrm{2}}$.
    Since the diagonal group $a=b$ acts trivially, this
    stack is $(\tilde P/\gm)\times B(\gm)$. Clearly,
    $\tilde P/\gm=P$.
\end{itemize}
\end{proof}

\begin{Remark}\label{RemP}
Note that we can (and shall) view $P$ as the moduli space of
collections
\[
    (L,\eta,O_\p1(-2)\hookrightarrow L_\eta).
\]
\end{Remark}

\begin{proof}[Proof of Theorem~\ref{StackProp}]
By Propositions~\ref{Affine}(\ref{Aff2}) and~\ref{BunP} $\cM$
is a smooth connected algebraic stack of dimension $1$. To
prove that $\cM=M\times B(\gm)$ consider the moduli stack of
triples $(L,\nabla,O_{\p1}(-2)\hookrightarrow L_\eta)$, where
$(L,\nabla)\in\cM$; denote it by $M'$. We have $\cM=M'/\gm$, where $\gm$ acts by
rescaling the embedding $O_{\p1}(-2)\hookrightarrow L_\eta$. It
is easy to see that this action is trivial, so that
$\cM=M'\times B(\gm)$. On the other hand,
Lemma~\ref{Irreducible} shows that connections in $\cM$ have
only scalar automorphisms, thus $M'$ is an algebraic space.
Therefore $\cM\to M'$ is the good moduli space and $M'=M$ by
uniqueness of good moduli spaces. We see that $\cM$ is a
neutral gerbe over $M$. Next, we have a cartesian diagram
\[
\begin{CD}
M @>>> \cM\\
@VVV @VV \rho V\\
P @>>> \Bun(d).
\end{CD}
\]
Thus $M$ is an affine bundle over $P$. It follows that $M$ is a
smooth surface. Finally, $M$ is quasi-projective, since it is
open in $\oM$.
\end{proof}

\begin{proof}[Proof of Proposition~\ref{CohomologicalDimension}]
The map $\rho:\cM\to\Bun(d)$ is an affine bundle, thus it is an
affine morphism. On the other hand, the good moduli space of
$\Bun(d)$ is $P$, which is a 1-dimensional scheme.
\end{proof}

\subsection{$\Mb$ is a locally complete intersection}
\begin{Lemma}\label{mnOne}
If $(L,\nabla;\epsilon\in E)\in\Mb$, then $L\simeq
O_\p1(n)\oplus O_\p1(m)$ with $m-n=1$.
\end{Lemma}
\begin{proof}
For $\epsilon\ne0$ this is
Proposition~\ref{Affine}~(\ref{Aff0}). Let $\epsilon=0$ and
assume that $m-n\ge3$. Then the same argument as in
Proposition~\ref{Affine}~(\ref{Aff0}) shows that $O_\p1(m)$ is
$\nabla$-invariant, which contradicts semistability.
\end{proof}

\begin{Proposition}\label{LocCI}
$\Mb$ is a locally complete intersection.
\end{Proposition}
\begin{proof}
Let $\divisor'\supset\divisor$ be a divisor on $X$. Consider
the moduli stack
$\widetilde\cN(\divisor')\subset\Nb(\p1,2,d,\divisor')$
parameterizing $(L,\nabla;\epsilon\in E)$, where $L\simeq
O_\p1(n)\oplus O_\p1(m)$ with $m-n=1$, $m+n=d$, $\nabla:L\to
L\otimes\Omega_\p1(\divisor')\otimes_\C E$ is an
$\epsilon$-connection, $(L,\nabla)$ is semistable and
non-nilpotent. It is enough to show that if $\deg\divisor'$ is
big enough, then (i) $\widetilde\cN(\divisor')$ is smooth and
(ii) $\Mb$ is defined by $\dim\widetilde\cN(\divisor')-\dim\Mb$
equations in $\widetilde\cN(\divisor')$ (note that
$\Mb\subset\widetilde\cN(\divisor')$ by Lemma~\ref{mnOne}).

For (i) it is enough to show that the map
$\widetilde\cN(\divisor')\to\A1/\gm$ sending
$(L,\nabla;\epsilon\in E)$ to $\epsilon\in E$ is smooth. The
relative deformation complex of this map at
$(L,\nabla;\epsilon\in E)$ is
\[
    \cG^\al:=
    (\END(L)\xrightarrow{\ad\nabla}\END(L)\otimes
    \Omega_\p1(\divisor')\otimes_\C E),
\]
so that the obstruction to smoothness lies in the hypercomology group
$\HH^2(\p1,\cG^\al)$, which is a quotient of \[
H^1(\p1,\END(L)\otimes\Omega_\p1(\divisor')\otimes_\C E).
\]
The latter group vanishes for $\deg\divisor'$ large enough because
$L\simeq O_\p1(n)\oplus O_\p1(m)$, with $m-n=1$.

For (ii) note that
$\dim\widetilde\cN(\divisor')=-\chi(\cG^\al)=4\deg\divisor'-8$.

In Corollary~\ref{MhY} below, we shall give an explicit
description of $\Mh$; this description implies that
$\dim\Mh=0$. Combining this with Theorem~\ref{StackProp}, we
see that $\dim\Mb=1$.

Further, let $\cL$ be the vector bundle on
$\widetilde\cN(\divisor')$ whose fiber at
$(L,\nabla;\epsilon\in E)$ is $\END(L)\otimes
(O_\p1(\divisor')/O_\p1(\divisor))\otimes_\C E$. The polar part
of $\nabla$ gives rise to a section of $\cL$ and
$\widetilde\cN(\divisor)\subset\widetilde\cN(\divisor')$ is
given by the zero locus of this section. Thus
$\widetilde\cN(\divisor)$ is locally cut out by
$4(\deg\divisor'-\deg\divisor)$ equations. It follows from the
definition of $\Mb$ (cf. Section~\ref{DEFMB}) that $\Mb$ is cut out
from $\widetilde\cN(\divisor)$ by $2\deg\divisor-1$ equations
(note that the sum of residues of $\nabla$ is equal to $-d$).
\end{proof}

\begin{proof}[Proof of Proposition~\ref{Cartier}]
$\Mh$ is given by $\epsilon=0$, so we only need to check that
$\epsilon$ is locally not a zero divisor on $\Mb$. However, if
it was the case, $\Mh$ would contain a component of $\Mb$
(set-theoretically), and we would come to a contradiction with
complete intersections having pure dimension.
\end{proof}

\subsection{Universal Moduli Spaces} Recall that
$\divisor=\sum n_ix_i$. Fix $\infty\in\p1\setminus\divisor$.

Consider a moduli space $\cB$, parameterizing local invariants
of connections, that is, triples $(\divisor,\nu_1,\nu_2)$,
where $\divisor$ is a degree 4 divisor on $\p1$ such that
$\infty\notin\supp\divisor$,
$\nu_1\in\Omega_\p1(\divisor)/\Omega_\p1$,
$\nu_2\in\Omega^{\otimes2}_\p1(2\divisor)/\Omega^{\otimes2}_\p1(\divisor)$,
and the sum of residues of $\nu_1$ equals $-d$,
cf. Section~\ref{DEFMB}. We can identify such $\divisor$ with roots
of degree 4 monic polynomial $p(z)$, then we can write uniquely
\[
\nu_1=\frac{a_0+a_1z+a_2z^2+dz^3}{p(z)}\,\dif z,\qquad
\nu_2=\frac{b_0+b_1z+b_2z^2+b_3z^3}{p(z)^2}\,\dif z\otimes\dif z.
\]
Here $z$ is the standard coordinate on $\p1$. Thus
$\cB\simeq\C^{11}$. Note that the subset of points in $\cB$
satisfying conditions of Section~\ref{MODST} is open in
\emph{analytic\/} topology.

As $(\divisor,\nu_1,\nu_2)$ varies, we obtain a family
$\Mb_{univ}\to\cB$ of moduli stacks. In Section~\ref{DEFMB} we fixed
$\divisor$, $\nu_1$ and $\nu_2$. Denote the corresponding point
of $\cB$ by $t_0$. Then the fiber of $\Mb_{univ}\to\cB$ over
$t_0$ is $\Mb$. Our goal is to prove
\begin{Proposition}\label{PrFlat}
The family $\Mb_{univ}\to\cB$ is a flat family of stacks in a
Zariski neighborhood of $t_0\in\cB$.
\end{Proposition}
\begin{proof}
Similarly to the previous subsection we prove that $\Mb_{univ}$
is a locally complete intersection of dimension $\dim\cB+1$. It
follows that the fibers of $\Mb_{univ}\to\cB$ are at least
1-dimensional. By semicontinuity, there is a neighborhood of
$t_0$, where fibers are 1-dimensional. It remains to note, that
by~\cite[Proposition~6.15]{EGAIV} a morphism from a locally
complete intersection to a smooth scheme with equidimensional
fibers is flat.
\end{proof}

\section{Generalized line bundles on generalized elliptic
curves}\label{GENELL}

In this section, we review the Fourier-Mukai
transform for singular degenerations of elliptic curves.
Such transform is constructed in \cite{EllipticFM};
the case of singular reduced irreducible genus one curve
(nodal or cuspidal) goes back to \cite{BurbanKreussler}.

As explained in \cite{EllipticFM}, the Fourier-Mukai transform
provides isomorphisms between various moduli spaces of semistable
sheaves on the curve. In particular, the moduli space of semistable
generalized line bundles of degree $-1$ is naturally identified with
the curve itself. (The precise definitions are given below.)
For our purposes, it is desirable to work with moduli stacks rather
than the moduli spaces: we outline the required (very straightforward)
changes to the results of \cite{EllipticFM} below.

For the reader's convenience, we provide sketches of proofs along with
appropriate references.

\subsection{Generalized elliptic curves}

For the purposes of this paper, it is important to work with
all double covers of $\p1$ ramified at four points, including
reducible covers (see Remark~\ref{Types}). However, the
argument naturally applies to the following class of curves.

\begin{Definition}
A projective curve $Y$ is \emph{generalized elliptic} if
$H^0(Y,O_Y)=\C$ (in particular, $Y$ is connected and has no
embedded points) and the dualizing sheaf of~$Y$ is trivial. In
particular, $Y$ is Gorenstein and has arithmetic genus $1$.
\end{Definition}

\begin{Remark}\label{RemDualizing}
The dualizing sheaf of $Y$ is canonically identified with
\[O_Y\otimes_\C H^1(Y,O_Y)^\vee\]
Indeed, its space of global sections is identified with $H^1(Y,O_Y)^\vee$ by
Serre's duality.
\end{Remark}

\begin{Example}
Any plane cubic (reduced or not) is a generalized elliptic
curve.
\end{Example}

\begin{Remark}
Note that the singularities of $Y$ need not be planar.
For example, an intersection of two space quadrics is a
generalized elliptic curve, even if the two quadrics are cones
with a common vertex. In this case, the intersection is a union
of four lines that meet at the vertex.
\end{Remark}

We need to consider (semi)stable coherent sheaves on $Y$.
These can be defined using the Hilbert function corresponding to some
polarization. Since $Y$ is a curve, we can use the following equivalent
definition.

Denote by $\Sigma$ the collection of generic points of $Y$ (by
definition a point $s\in Y$ is generic if its local ring
$O_{Y,s}$ is Artinian). For a sheaf $\ell$ on $Y$ and $s\in
\Sigma$, we denote by $\rk_s\ell$ the length of the stalk
$\ell_s$ as a module over the local ring $O_{Y,s}$. In
particular, $m(s):=\rk_sO_Y$ is the multiplicity of the
corresponding irreducible component.

Fix a weight function $w:\Sigma\to\Q^{>0}$ and set
\[
    \rk\ell=\rk_w\ell:=\sum_{s\in\Sigma}w(s)\rk_s\ell.
\]
We can now use this notion of rank (and the corresponding
notion of slope) to define a stability of coherent sheaves on
$Y$.

\begin{Definition}
A coherent sheaf $\ell\ne 0$ of pure dimension 1 is said to be
\emph{semistable} if for any proper subsheaf
$\ell_0\subset\ell$, $\ell_0\ne0,\ell$, we have
\[
\frac{\chi(\ell)}{\rk\ell}\ge\frac{\chi(\ell_0)}{\rk\ell_0}.
\]
If the inequality is strict, $\ell$ is \emph{stable}.
\end{Definition}

\begin{Definition}
We say that a sheaf $\ell$ is a \emph{generalized line bundle}
on $Y$ if it is of pure dimension 1 and its length at all
generic points of $Y$ equals to the multiplicity of the
corresponding component: $\rk_s\ell=m(s)$ for $s\in\Sigma$.

By definition, $\deg\ell:=\chi(\ell)-\chi(O_Y)=\chi(\ell)$.
\end{Definition}

Denote by $\CPic{d}(Y)$ the stack of generalized line bundles of
degree $d$ on $Y$, and let
$\CPic{d}_s(Y)\subset\CPic{d}_{ss}(Y)\subset\CPic{d}(Y)$ stand for
the open substacks of stable and semistable generalized line
bundles, respectively.

\subsection{The Fourier-Mukai transform}\label{FMproof}

Let $\cP:=O_{Y\times Y}(-\Delta)$ be the ideal sheaf of the
diagonal $\Delta\subset Y\times Y$. Note that~$\cP$ is flat
over both factors, being the kernel of a surjection of
sheaves that are flat over both factors.
Generally speaking, $\cP$ is not flat
over the product.

Consider the Serre dual $\cP^\vee:=\HOM(\cP,O_{Y\times
Y})$ of $\cP$. The proof of \cite[Proposition~1.3(3)]{EllipticFM}
shows that $\cP^\vee$ is of finite Tor-dimension over both factors.
(Actually, it is easy to see that $\cP^\vee$ is isomorphic to an extension of $O_\Delta$ by $O_{Y\times Y}$,
and therefore flat over both factors.)

\begin{Proposition}{\cite[Proposition~1.3]{EllipticFM}}
\label{FourierMukai}
The Fourier-Mukai transform with the kernel $\cP$
\[\Xi:\cD^b(Y)\to\cD^b(Y):\cF\mapsto Rp_{1,*}(\cP\otimes^Lp_2^*\cF)\]
is an auto-equivalence of the category of $\cD^b(Y)$ \lpar{}the
bounded derived category of quasi-coherent sheaves on $Y$\rpar{}.
Here $p_1,p_2:Y\times Y\to Y$ are projections.

The inverse of this equivalence is given by
\[
\Xi':\cD^b(Y)\to\cD^b(Y):\cF\mapsto Rp_{2,*}(\cP^\vee\otimes^Lp_1^*\cF)
    \otimes_\C(H^1(Y,O_Y))^{-1}[1].
\]
\end{Proposition}

\begin{proof}[Sketch of proof]
Proposition~\ref{FourierMukai} is proved in \cite{EllipticFM}
(if $C$ is irreducible, see also \cite[Proposition~2.10]{BurbanKreussler}).
The key idea is that the structure sheaf $O_Y\in\cD^b(Y)$ is a \emph{spherical object} in the sense of
Seidel and Thomas \cite{SeidelThomas}. A spherical object $\cE\in\cD^b(Y)$ defines
an equivalence $T_\cE:\cD^b(Y)\to\cD^b(Y)$ called the \emph{twist
functor}. For the spherical
object $\cE=O_Y$, the twist functor $T_\cE$ is isomorphic to $\Xi$. Therefore, $\Xi$ is an equivalence.

By Serre's duality, $\Xi$ and $\Xi'$ are adjoint. Therefore, they are mutual inverses.
\end{proof}

\begin{remark*} \cite{EllipticFM} considers the Fourier-Mukai transform on the bounded
coherent derived category $\cD^b_{coh}(Y)$. However, the argument works equally well for
the derived category $\cD^b(Y)$ (or for the unbounded derived category $\cD(Y)$).

Compared to \cite{EllipticFM}, we introduced a twist by the one-dimensional
vector space $(H^1(Y,O_Y))^{-1}$ in the formula for $\Xi'$. This makes the isomorphism between
$\Xi'$ and the inverse of $\Xi$ canonical.
\end{remark*}

Let us translate Proposition~\ref{FourierMukai} into the language of
Fourier-Mukai kernels, since we shall use
Proposition~\ref{FMOrthog} later. Let $p_{12}$, $p_{13}$, and
$p_{23}$ be the usual projections $Y\times Y\times Y\to Y\times
Y$. Set $\cF_Y:=Rp_{13,*}(p^*_{12}\cP^\vee\otimes^L
p^*_{23}\cP)$. Recall that both
$\cP$ and $\cP^\vee$ are flat with respect to both projections
$Y\times Y\to Y$. In particular,
\[p^*_{12}\cP^\vee\otimes^L p^*_{23}\cP=p^*_{12}\cP^\vee\otimes p^*_{23}\cP
\] is a sheaf, so that $\cF_Y$ is concentrated in
cohomological dimensions $0$ and $1$.

\begin{Proposition}\label{FMOrthog}
\[
   \cF_Y=O_\Delta[-1]\otimes_\C H^1(Y,O_Y)\simeq O_\Delta[-1].
\]
\end{Proposition}
\begin{proof} Let
\[\Xi_0:D^b(Y)\to D^b(Y):\cF\mapsto Rp_{2,*}(\cF_Y\otimes^Lp_1^*\cF)\otimes_\C(H^1(Y,O_Y))^{-1}[1]\]
be the Fourier-Mukai transform with the kernel $\cF_Y\otimes_\C(H^1(Y,O_Y))^{-1}[1]$.
There is a canonical isomorphism between $\Xi_0$ and the composition $\Xi\circ\Xi'$, which
is isomorphic to the identity functor by Proposition~\ref{FourierMukai}.

An isomorphism between $\Xi_0$ and the identity functor yields an isomorphism between the kernel of $\Xi_0$
and $O_\Delta$. Unfortunately, most references for this statement impose extra assumptions, such as
smoothness. For the sake of completeness, let us sketch the proof.

For every point $y\in Y$, consider the length-one sky-scraper sheaf $O_y$ supported at $y$.
Since $\Xi_0(O_y)\simeq O_y$, it follows from
the proof of \cite[Corollary~1.12]{Nahm} that
$\cF_Y\otimes_\C(H^1(Y,O_Y))^{-1}[1]$ is a line bundle on the graph of a morphism $Y\to Y$.
More precisely, there is a map $\psi:Y\to Y$ and a line bundle
$L$ on $Y$ such that
\[\cF_Y\otimes_\C(H^1(Y,O_Y))^{-1}[1]\simeq (\id_Y,\psi)_*L.\]

Therefore, $\Xi_0$ is isomorphic to the functor
$\cF\mapsto R\psi_*(L\otimes^L\cF)$.
On the other hand, $\Xi_0$ is the identity functor. It remains to verify that $\psi=\id_Y$ and $L\simeq O_Y$.

Indeed, looking at the action of $\Xi_0$ on objects $O_y$, we see that $\psi$ induces the identity map on the
set of points of $Y$. Looking at the action of $\Xi_0$ on endomorphisms of objects $j_*(O_U)$ for open
embeddings $j:U\hookrightarrow Y$, we conclude that $\psi=\id_Y$. Since $\Xi(O_Y)=L$, we see that $L\simeq O_Y$,
as claimed.
\end{proof}

\subsection{Stable generalized line bundles}

Since the coherent sheaf $\cP$ on $Y\times Y$ is flat over the first factor, we can view it as a $Y$-family of coherent sheaves on $Y$. This family assigns to
every $y\in Y$ its ideal sheaf $O_Y(-y)$. Obviously, $O_Y(-y)$ is a generalized line bundle of degree $-1$ on $Y$ for every $y\in Y$.

\begin{Proposition} \label{Pointwise iso}
For every $y\in Y$, the generalized line bundle $O_Y(-y)$ is stable. Conversely, any semistable generalized line bundle of degree $-1$ is isomorphic
to $O_Y(-y)$ for unique $y\in Y$.
\end{Proposition}
\begin{proof}
The claim easily follows from  Proposition~1.9(1) and Theorem~1.20 of \cite{EllipticFM} (see also \cite[Corollary~1.23]{EllipticFM}). Indeed, these results imply
that for $\cF\in D^b(Y)$, the following conditions are equivalent
\begin{itemize}
\item $\cF$ is a stable (resp. semistable) coherent sheaf with $\rk_w\cF=0$ and\\ $\chi(\cF)=1$;

\item Its Fourier-Mukai transform $\Xi(\cF)$ is a stable (resp. semistable) coherent sheaf with $\rk_w\cF=\rk_w O_Y$ and $\chi(\cF)=-1$.
\end{itemize}
Note however that the first condition simply requires that $\cF$ is a sky-scraper sheaf of length one, and that any such sheaf is automatically stable. Since
$\Xi(O_y)=O_Y(-y)$, the claim follows.
\end{proof}

\begin{Remark}
In particular, the moduli stack $\CPic{-1}_s(Y)=\CPic{-1}_{ss}(Y)$ does not
depend on $w$.
\end{Remark}

Proposition~\ref{Pointwise iso} gives a bijection between points of $\CPic{-1}_s(Y)$ and
points of $Y$, which are in bijection with isomorphism classes of sky-scraper sheaves of length one on $Y$.
This bijection naturally extends to an isomorphism between the corresponding moduli stacks. If one
works with the moduli spaces instead, this is a special case of \cite[Corollary~1.25]{EllipticFM}.

Let us go over the construction. By Proposition~\ref{Pointwise iso}, $\cP$ is a $Y$-family of stable
degree $-1$ generalized line bundles on $Y$, so it defines a map
$Y\to\CPic{-1}_s(Y)$.  The above map naturally extends to a map
\begin{equation}
Y\times B(\gm)\to\CPic{-1}_s(Y).\label{picmap}
\end{equation}
Explicitly, for a test scheme $S$, the map \eqref{picmap}
assigns to a morphism $\psi:S\to Y$ and a line bundle $L$ on
$S$ (recall that a line bundle on $S$ is the same as a map
$S\to B(\gm)$) the sheaf $p_1^*L\otimes O_{S\times
Y}(-\Gamma_\psi)$ on $S\times Y$, viewed as an $S$-family of
degree $-1$ generalized line bundles on $Y$, that is, as a
morphism $S\to\CPic{-1}(Y)$. Here $\Gamma_\psi\subset S\times
Y$ is the graph of $\psi$.

\begin{Proposition}\label{PicY}
The map \eqref{picmap} is an isomorphism
\[
    Y\times B(\gm)\iso\CPic{-1}_s(Y)=\CPic{-1}_{ss}(Y).
\]
\end{Proposition}

\begin{proof} Let $S$ be a Noetherian test scheme. A morphism $S\to Y\times B(\gm)$ is a pair
of a morphism $\psi:S\to Y$ and a line bundle $L$ on $S$. To such pair, we assign the coherent sheaf
$\cF_{\psi,L}=(\id_S,\psi)_*L$ on $S\times Y$. Obviously, $\cF_{\psi,L}$ is a family of sky-scraper sheaves
of length one on $Y$: it is flat over $S$ and its restriction to $\{s\}\times Y$ is a length-one sky-scraper
sheaf for any $s\in S$. Conversely, any family of sky-scraper sheaves of length one on $Y$ is of the form
$\cF_{\psi,L}$ for uniquely determined $\psi$ and $L$. In other
words, as $S$ varies, the correspondence
\[(\psi,L)\mapsto \cF_{\psi,L}\]
defines an isomorphism between $Y\times B(\gm)$ and the moduli stack of sky-scraper sheaves of length one on $Y$.

Proposition~\ref{Pointwise iso} implies that the Fourier-Mukai transform defines an equivalence between
the groupoid of $S$-families of sky-scraper sheaves of length one on $Y$ and the groupoid of $S$-families
of degree $-1$ stable generalized line bundles on $Y$. This follows from Corollary~1.9 of \cite{Nahm}
(see also Proposition~4.2 of \cite{EquivalencesAndFM}, or, for the classical Fourier-Mukai transform on abelian varieties, \cite[Theorem~1.6]{FourierAndModuli}). Explicitly, if $\cF$ is an $S$-flat coherent sheaf on $S\times Y$ whose restriction to every fiber $\{s\}\times Y$ is a length-one sky-scraper sheaf, the corresponding family of
generalized line bundles is
\[\Xi(\cF)=Rp_{13,*}(p_{23}^*(\cP)\otimes^L p_{12}^*(\cF)).\]
Here $p_{12},p_{13}:S\times Y\times Y\to S\times Y$ and $p_{23}:S\times Y\times Y\to Y\times Y$ are the
projections.

It now remains to notice that for $\cF=\cF_{\psi,L}$, its image $\Xi(\cF_{\psi,L})$ is given by \eqref{picmap}.
\end{proof}

\section{Geometric description of $\Mh$}\label{MH}
Recall that our goal is to calculate cohomology of certain
natural vector bundle on $\cM$ (or more precisely, a direct
image, see Theorem~\ref{Theorem3}). In this section we
calculate the direct image of the extension of this sheaf to
$\Mh$ (see Section~\ref{DEFMB} for the definition of $\Mh$). The main
result is Proposition~\ref{MhXxY}. The calculation is based on
explicit identification of $\Mh$, see Corollary~\ref{MhY}, and
applying the Fourier--Mukai transform.

We claim that $\Mh$ is the moduli stack of collections
$(L,\nabla;E)$, where $L$ is a rank 2 degree $d$ vector bundle
on $\p1$, $E$ is a one-dimensional vector space, $\nabla:L\to
L\otimes\Omega_\p1(\divisor)\otimes_\C E$ is an
$O_{\p1}$-linear morphism, satisfying the following
conditions:\\
\noindstep\label{NoNilp} $\nabla$ is not nilpotent, that is $\nabla^2\ne0$.\\
\noindstep\label{HitchinTrace} $\tr\nabla=0$.\\
\noindstep\label{HitchinDet} $\det\nabla$ is a section of
$E^{\otimes2}\otimes_\C\Omega_\p1^{\otimes2}(\divisor)$.\\
\noindstep $(L,\nabla;E)$ is semistable.\\
Note that $\tr\nabla$ is a section of
$E\otimes_\C\Omega_\p1(\divisor)$. It follows
from~(\ref{hitchincond}) that $\tr\nabla$ is in fact a section
of $E\otimes_\C\Omega_\p1$, which implies
condition~(\ref{HitchinTrace}). Condition~(\ref{HitchinDet}) is
a condition on the polar part of $\nabla$: a priori
$\det\nabla$ is in
$E^{\otimes2}\otimes_\C\Omega_\p1^{\otimes2}(2\divisor)$. This
condition also follows from~(\ref{hitchincond}).

Note that $\nabla^2=-\det\nabla\otimes\id_L$. Recall that $\cE$
is the line bundle on $\Mb$ whose fiber at
$(L,\nabla;\epsilon\in E)$ is $E$. For simplicity we write
$\cE$ for $\cE|_\Mh$. The following statement follows
from~(\ref{NoNilp}) and~(\ref{HitchinDet}) above
\begin{Lemma}\label{Esquare}
\[
\cE^{\otimes2}|_\Mh\simeq O_{\Mh}.
\]
\end{Lemma}
Let us fix a global section $\mu$ of
$\Omega_\p1^{\otimes2}(\divisor)\simeq O_{\p1}$, $\mu\ne0$. One
can choose an isomorphism $E\simeq\C$ such that
$\det\nabla=\mu$ (there are two choices for such an
isomorphism). Denote by $\cY$ the moduli stack of pairs
$(L,\nabla)$, where $L$ is a rank 2 degree $d$ vector bundle on
$\p1$, $\nabla\in H^0(\p1,\END(L)\otimes
\Omega_\p1(\divisor))$, $\tr\nabla=0$, $\det\nabla=\mu$, and
the pair $(L,\nabla)$ is semistable. We have proved the
following
\begin{Proposition}
The correspondence $(L,\nabla)\mapsto(L,\nabla;0\in\C)$ yields
a double cover $\cY\to\Mh$. Besides, $\Mh$ is identified with
the quotient stack $\mu_2\backslash\cY$, where $\pm1\in\mu_2$
acts on $\cY$ by $(L,\nabla)\mapsto(L,\pm\nabla)$.
\end{Proposition}

It follows directly from the definition of $\cY$ that the
pullback of $\cE$ to $\cY$ is $O_\cY$.

Set $\cA:=O_\p1\oplus\Omega_\p1(\divisor)^{-1}$. Then $\cA$ is
a sheaf of $O_\p1$-algebras with respect to the multiplication
\[
    (f_1,\tau_1)\times(f_2,\tau_2):=
(f_1f_2-\mu\otimes\tau_1\otimes\tau_2,f_1\tau_2+f_2\tau_1).
\]
Set $\pi:Y:= \SPEC(\cA)\to\p1$. Denote by $y_i\in Y$ the
preimage of $x_i\in\p1$, and by $\sigma:Y\to Y$ the involution
induced by $\sigma^*:\cA\to\cA:(f,\tau) \mapsto(f,-\tau)$.

\begin{Proposition}
$Y$ is a generalized elliptic curve.
\end{Proposition}
\begin{proof}
Since $\pi$ is a finite morphism, $Y$ has dimension 1. The
dualizing complex of $Y$ is given by $\HOM(\cA,\Omega_\p1)$.
Thus we need to show that this sheaf is isomorphic to $\cA$ as
an $\cA$-module. It is clear on the level of $O_\p1$-modules,
since $\cA\simeq O_\p1\oplus O_\p1(-2)$. Let
$\gamma\in\HOM(\cA,\Omega_\p1)$ be the composition of the
projection $\cA\to(\Omega_\p1(\divisor))^{-1}$ and an
isomorphism. One checks easily that the map of $\cA$-modules
$\cA\to\HOM(\cA,\Omega_\p1)$ given by $1\mapsto\gamma$ is
injective. Now, an injective map of a vector bundle to an
isomorphic one is necessarily an isomorphism.

Also, $H^0(Y,O_Y)=H^0(\p1,\cA)=\C$, thus $Y$ is generalized
elliptic.
\end{proof}

\begin{Remark}\label{Types}
Actually $Y$ is always reduced. Precisely, $Y$ is a smooth
elliptic curve if $\divisor$ has no multiple points; $Y$ is a
nodal cubic if $\divisor$ has a single multiple point of
multiplicity 2; $Y$ is a cuspidal cubic if
$\divisor=3(x_1)+(x_2)$; $Y$ has two components, isomorphic to
$\p1$, which intersect transversally at two points if
$\divisor=2(x_1)+2(x_2)$; and $Y$ has two components,
isomorphic to $\p1$, which are tangent to each other if
$\divisor=4(x_1)$.
\end{Remark}

\begin{Proposition}\label{Hitchin}
$\cY$ is naturally isomorphic to $\CPic{d+2}_sY$, that is the
moduli stack of stable generalized line bundles of degree $d+2$
on $Y$.
\end{Proposition}
\begin{proof}
Let $(L,\nabla)$ be a point of $\cY$. Then $L$ is an
$\cA$-module with respect to the multiplication
$(f,\tau)s:=fs+\tau\otimes\nabla s$, let us denote the
corresponding sheaf on $Y$ by $\ell$. It is a standard
fact about the Hitchin system that $\ell$
is a generalized line bundle on $Y$. The inverse construction
is given by $\ell\mapsto L:=\pi_*\ell$.

Let the weight function $w$ from Section~\ref{GENELL} be given by the
degree of the projection $\pi:Y\to\p1$. Then
$\rk\pi_*\ell_0=\rk\ell_0$ for any coherent sheaf $\ell_0$ on
$Y$.

We would like to show that $\ell$ is stable if and only if
$(L,\nabla)$ is stable. Note that $\nabla$-invariant subsheaves
of $L$ are in bijection with subsheaves $\ell_0\subset\ell$ via
$\ell_0\mapsto\pi_*\ell_0$. Further,
\begin{equation}\label{degree}
\deg\ell_0=\chi(\ell_0)=\chi(\pi_*
\ell_0)=\deg\pi_*\ell_0+\rk\pi_*\ell_0.
\end{equation}
It follows that the stability condition is the same.

It also follows from~(\ref{degree}) that the generalized line
bundles on $Y$ corresponding to rank 2 degree $d$ bundles on
$\p1$ have degree $d+2$.
\end{proof}

\begin{Remark}\label{SsTrivial}
If $\divisor$ is not even, then $Y$ is integral, and for every
$\ell_0\subset\ell$, $\ell_0\ne0$ we have $\rk\ell_0=\rk\ell$
thus the semistability condition is trivial.
\end{Remark}

Fix a degree $(d+3)/2$ line bundle $\lb$ on $\p1$ (recall that
$d$ is odd). A Higgs bundle $(L,\nabla)$ is semistable if and
only if $(L\otimes\lb,\nabla)$ is. Therefore,
Proposition~\ref{PicY} implies the following

\begin{Corollary}\label{MhY} Consider the map $Y\to\cY$ that
sends $y\in Y$ to the vector bundle $\lb\otimes\pi_* O_Y(-y)$
equipped with the natural Higgs field. The map induces an
isomorphism
\[Y\times B(\gm)\iso\cY.\]
Thus $\Mh$ is the quotient of the generalized elliptic curve
$Y$ by the action of $\mu_2\times\gm$, where $\mu_2$ acts by
$\sigma$, $\gm$ acts trivially.\qed
\end{Corollary}

 Let us use the isomorphism of Corollary~\ref{MhY} to describe
the universal Higgs bundle on $\p1\times\Mh$. Denote this
universal Higgs bundle by $\xi$ and its pullback to $\p1\times
Y$ by $\tilde\xi:=(\id_\p1\times\bar\pi)^*\xi$ (here $\bar\pi$
is the natural composition $Y\to\cY\to\Mh$). Recall also that
$\cP:=O_{Y\times Y}(-\Delta)$ is the ideal sheaf of the
diagonal.

\begin{Corollary}\label{FMcalculation} We have
$\tilde\xi=(\pi\times\id_Y)_*\cP\otimes p_1^*\lb$. The action
of $\gm$ on $\tilde\xi$ is via the identity character $a\mapsto
a$ and the action of $\mu_2$ comes from its action on $\cP$
\lpar{}on $Y\times Y$, $-1\in\mu_2$ acts as
$\sigma^*\times\sigma^*$\rpar.

For the dual bundle,
\[
\tilde\xi^\vee=(\pi\times\id_Y)_*\cP^\vee\otimes
p_1^*(\lb^\vee\otimes\cT_\p1)\otimes_\C H^1(Y,O_Y)^\vee.
\]
On this bundle, $\gm$ acts via the character $a\mapsto a^{-1}$
and the action of $\mu_2$ comes from its action on $\cP$ and
its action on $H^1(Y,O_Y)$ \lpar{}by $-1$\rpar.
\end{Corollary}
\begin{proof}
The description of $\tilde\xi$ follows from
Proposition~\ref{Hitchin} and Corollary~\ref{MhY}, and then the
description of $\tilde\xi^\vee$ follows from Serre duality.
\end{proof}

\begin{Remark}
It is easy to describe the Fourier-Mukai transform of
$\tilde\xi$: this is the structure sheaf of the graph of $\pi$
twisted by $\lb$.
\end{Remark}

Consider now the sheaf $\cF_H:=p_{13}^*\xi\otimes
p_{23}^*\xi^\vee$ on $\p1\times\p1\times\Mh$. The main result
of this section is the following

\begin{Proposition}\label{MhXxY}
\[
    Rp_{12,*}(\cF_H\otimes p_3^*\cE^{\otimes k})\simeq
    \begin{cases}
        \iota_{\Delta,*}(\cT_\p1)[-1]&\text{if $k$ is even},\\
        \iota_{\Delta,*}(\cT^{\otimes 2}_\p1(-\divisor))[-1]&\text{if $k$ is odd},
    \end{cases}
\]
where $\iota_\Delta:\p1\to\p1\times\p1$ is the diagonal
embedding.
\end{Proposition}
\begin{proof}
The pullback
$(\id_{\p1\times\p1}\times\bar\pi)^*(p_{13}^*\xi\otimes
p_{23}^*\xi^\vee)$ to $\p1\times\p1\times Y$ equals
$p_{13}^*\tilde\xi\otimes p_{23}^*\tilde\xi^\vee$. By
Corollary~\ref{FMcalculation} and Proposition~\ref{FMOrthog},
we have the following identity on $\p1\times\p1$
\begin{multline*}
Rp_{12,*}(p_{13}^*\tilde\xi\otimes p_{23}^*\tilde\xi^\vee)=
\\(\pi\times\pi)_*Rp_{12,*}(p_{13}^*\cP\otimes p_{23}^*\cP^\vee)
\otimes p_1^*\lb\otimes
p_2^*(\lb^\vee\otimes\cT_\p1)\otimes_\C H^1(Y,O_Y)^\vee=\\
\iota_{\Delta,*}(\pi_*O_Y\otimes\cT_\p1)[-1].
\end{multline*}
The action of $\gm$ on the right-hand side is trivial. The
action of $\mu_2$ on $O_Y$ is the standard action coming from
$\sigma:Y\to Y$ (in other words, $-1\in\mu_2$ acts by
$\sigma^*$).

Note that $\bar\pi^*\cE=O_Y$, but $-1\in\mu_2$ acts on
$\bar\pi^*\cE$ as $-\sigma^*$ (and $\gm$ acts trivially). Since
$\Mh=Y/(\gm\times\mu_2)$,
\begin{multline*}
Rp_{12,*}(\cF_H\otimes p_3^*\cE^{\otimes k})=
\Bigl(Rp_{12,*}((\id_{\p1\times\p1}\times\bar\pi)^*
(\cF_H\otimes p_3^*\cE^{\otimes k}))\Bigr)^{\gm\times\mu_2}\\
\simeq\begin{cases}(\iota_{\Delta,*}(\pi_*O_Y\otimes\cT_\p1))^{(1)}[-1]&\text{if $k$ is even,}\\
(\iota_{\Delta,*}(\pi_*O_Y\otimes\cT_\p1))^{(-1)}[-1]&\text{if $k$ is
odd}.\end{cases}
\end{multline*}
Here for a sheaf $\cV$ with an action of $\mu_2$, we denote by
$\cV^{(1)}$ (resp.\ $\cV^{(-1)}$) its eigensheaf on which
$-1\in\mu_2$ acts as 1 (resp. $-1$). Finally,
\[\pi_*O_Y=\cA=O_\p1\oplus\Omega_\p1(\divisor)^{-1},\]
and $-1\in\mu_2$ acts on $O_\p1$ as $1$ and on
$\Omega_\p1(\divisor)^{-1}$ as $-1$.
\end{proof}

\section{First orthogonality relation}\label{ORTHOGONALITY}
In this section, we prove Theorem~\ref{Theorem3}.

\subsection{}
Recall that $\cF_P=p_{13}^*\xi_\alpha\otimes
p_{23}^*\xi^\vee_\alpha$ is a quasi-coherent sheaf on $P\times
P\times\cM$ equipped with an action of $\Dmod_{P,\alpha}$ along
the first copy of $P$ and an action of $\Dmod_{P,-\alpha}$
along the second. Accordingly, the direct image
$Rp_{12,*}\cF_P$ is an object of the derived category of
$p_1^\al\Dmod_{P,\alpha}\circledast
p_2^\al\Dmod_{P,-\alpha}$-modules on $P\times P$.

Let $\iota_\Delta:P\to P\times P$ be the diagonal embedding.
Recall that $\delta_\Delta$ is a
$\Dmod_{P,\alpha}\boxtimes\Dmod_{P,-\alpha}$-module given by
$\iota_{\Delta,*}O_P$.
\begin{Lemma}\label{Shriek}
In the category of
$\Dmod_{P,\alpha}\boxtimes\Dmod_{P,-\alpha}$-modules we have
\[
    \delta_\Delta=\iota_{\Delta,*}O_P=\iota_{\Delta,!}O_P.
\]
\end{Lemma}
\begin{proof}
Let $\Delta$ be the diagonal in $P\times P$ and
$\overline\Delta$ be its closure. We can decompose
$\iota_\Delta$ as
\[
\Delta\xrightarrow{\iota_1}\overline\Delta\xrightarrow{\iota_2}P\times P.
\]
Since $\iota_2$ is a closed embedding, we have
$\iota_{2,*}=\iota_{2,!}$. Thus it is enough to show that
\[
\iota_{1,*}O_P=\iota_{1,!}O_P.
\]
Note that $\iota_1$ is an open embedding,
$\overline\Delta-\Delta$ consists of 8 points, and twists at
these points are given by $\pm(\alpha_i^+-\alpha_i^-)$. Now the
statement follows from conditions~(\ref{AlphaI})
and~\eqref{AlphaIV} of Section~\ref{MODST}. Note that
$\iota_{\Delta,*}$ and $\iota_{\Delta,!}$ are exact functors,
since $\iota_\Delta$ is an affine embedding.
\end{proof}

Further, the restriction $(\iota_\Delta\times\id_\cM)^*\cF_P$
is a quasi-coherent sheaf on $P\times\cM$ equipped with a
structure of a $\Dmod_P$-module. Recall that $\wp:P\to\p1$ is
the natural projection. It is easy to see that we have a
natural inclusion $\wp^*\xi\subset\xi_\alpha$ (see
Remark~\ref{ForDisc}). Thus, the identity automorphism of $\xi$
gives a horizontal section {\small
\[
    1\in H^0(P\times\cM,
    (\iota_\Delta\times\id_\cM)^*(p_{13}^*\wp^*\xi\otimes p_{23}^*\wp^*\xi^\vee))\subset
    H^0(P\times\cM,(\iota_\Delta\times\id_\cM)^*\cF_P).
\]}
We thus obtain a horizontal section of
$p_{1,*}(\iota_\Delta\times\id_\cM)^*\cF_P$, which can be
viewed as a morphism
\[
    O_P\to Rp_{1,*}(\iota_\Delta\times\id_\cM)^*\cF_P=
    \iota_\Delta^* Rp_{12,*}\cF_P
\]
in the derived category of $\Dmod_P$-modules (we use base
change). Finally, adjunction provides a morphism
\[
    \phi:\delta_\Delta[-1]=\iota_{\Delta,!}O_P[-1]
    \to Rp_{12,*}\cF_P
\]
(we are using Lemma~\ref{Shriek}). Note that the appearance of
the shift $[-1]$ is due to the fact that our inverse images are
$O$-module inverse images; from the point of view of
$\Dmod$-modules they should contain shifts.

Theorem~\ref{Theorem3} claims that $\phi$ is an isomorphism. We
derive Theorem~\ref{Theorem3} from two statements that are
proved later in this section.

\begin{Proposition}\label{AwayFromDiag} The direct image
$R(\wp\times\wp)_*Rp_{12,*}\cF_P$ vanishes outside the diagonal
in $\p1\times\p1$.
\end{Proposition}

\begin{Proposition}\label{SSupport} Consider the morphism
\[H^1(\phi):\delta_\Delta\to R^1p_{12,*}\cF_P\]
induced by $\phi$. Then its cokernel is such that
$(\wp\times\wp)_*\coker(H^1(\phi))$ is coherent.
\end{Proposition}
\begin{Remark}
Note that in Proposition~\ref{SSupport}, we consider naive (not
derived) direct image $(\wp\times\wp)_*$. Actually, higher
derived images $R^i(\wp\times\wp)_*\cG$ ($i>0$) vanish for any
$p_1^\al\Dmod_{P,\alpha}\circledast
p_2^\al\Dmod_{P,-\alpha}$-module $\cG$ (see
Remark~\ref{cBRem}\eqref{cBRem3}).
\end{Remark}

\begin{proof}[Proof of Theorem~\ref{Theorem3}]
By Proposition~\ref{CohomologicalDimension}, we see that
$R^ip_{12,*}\cF_P=0$ for all $i\ne 0,1$. Also,
$R^0p_{12,*}\cF_P$ vanishes at the generic point by
Proposition~\ref{AwayFromDiag}, which implies that
$R^0p_{12,*}\cF_P=0$. Thus $Rp_{12,*}\cF_P$ is concentrated in
cohomological dimension one. It remains to show that
$H^1(\phi)$ is an isomorphism.

By construction, $\phi\ne 0$. Since $\delta_\Delta$ is
irreducible as a $p_1^\al\Dmod_{P,\alpha}\circledast
p_2^\al\Dmod_{P,-\alpha}$-module, $H^1(\phi)$ is injective. Its
cokernel $\cF':=\coker(H^1(\phi))$ is a
$p_1^\al\Dmod_{P,\alpha}\circledast
p_2^\al\Dmod_{P,-\alpha}$-module such that
$(\wp\times\wp)_*\cF'$ is a coherent sheaf (by
Proposition~\ref{SSupport}) that vanishes generically (by
Proposition~\ref{AwayFromDiag}). It is now easy to see that
$\cF'=0$.

Indeed, consider a stratification of $P\times P$ by sets of the
form $\{(x_i^\pm,x_j^\pm)\}$,
$\{x_i^\pm\}\times(\p1-\divisor)$,
$(\p1-\divisor)\times\{x_i^\pm\}$, and $(\p1-\divisor)\times
(\p1-\divisor)$. We can now show that $\cF'$ vanishes on all
strata by descending induction on the dimension of strata.
\end{proof}

\subsection{Proof of Proposition~\ref{AwayFromDiag}}
\begin{Lemma}\label{DirectIm}
\[
    R(\wp\times\id_\cM)_*\xi_\alpha=\xi.
\]
\end{Lemma}
\begin{proof}
The sheaves are obviously identified on
$(\p1-\divisor)\times\cM$, so it remains to verify that this
identification extends to $\p1\times\cM$. It suffices to check
this on $D\times\cM$, where~$D$ is the formal neighborhood of
$x_i$. The restriction $\xi|_{D\times\cM}$
decomposes into a direct sum
$\xi^+\oplus\xi^-$ of one-dimensional bundles that are
invariant under the connection (that acts in the direction of
$D$). This can be viewed as a version of diagonalization
\eqref{fnf}.

The preimage $\wp^{-1} D$ is a union of two copies of $D$ glued
away from the center, and the restrictions of  $\xi_\alpha$ to
$\wp^{-1} D$ is of the form $(j_+\times\id_\cM)_*\xi_+\oplus
(j_-\times\id_\cM)_*\xi_-$, where $j_\pm:D\to \wp^{-1} D$ are
the embeddings of the two copies (see Remark~\ref{ForDisc}).
Since $\wp\circ j_\pm=\id_D$, the claim follows.
\end{proof}

Consider now the sheaf $\cF_\p1:=p_{13}^*\xi\otimes
p_{23}^*\xi^\vee$ on $\p1\times\p1\times\cM$.
\begin{Corollary}\label{DirectImage}
\[
    R(\wp\times\wp\times\id_\cM)_*\cF_P=\cF_\p1.
\]
\end{Corollary}
\begin{proof}
This follows from Lemma~\ref{DirectIm} and a similar statement
about $\xi_\alpha^\vee$ upon writing
$\cF_P=\Delta_{24}^*(\xi_\alpha\boxtimes\xi_\alpha^\vee)$,
where $\Delta_{24}:P\times P\times\cM\to P\times\cM\times
P\times\cM$ is a partial diagonal.
\end{proof}

By Corollary~\ref{DirectImage},
\begin{equation}\label{dirimage}
    R(\wp\times\wp)_*Rp_{12,*}\cF_P=Rp_{12,*}\cF_\p1.
\end{equation}
The advantage of working with $\xi$ rather than $\xi_\alpha$ is
that $\xi$ is naturally defined as a vector bundle (the
universal bundle) on $\p1\times\Mb$. Accordingly, $\cF_\p1$
extends to a vector bundle $\overline\cF:=p_{13}^*\xi\otimes
p_{23}^*\xi^\vee$ on $\p1\times\p1\times\Mb$. Set
\[
    \cF_k:=\overline\cF(k(\p1\times \p1\times\Mh)),\qquad k\in\Z.
\]
Let $\jmath:\p1\times\p1\times\cM\to\p1\times\p1\times\Mb$ be
the natural embedding. In view of Proposition~\ref{Cartier}, we
have a filtration
\begin{equation}\label{filtration}
\cF_0=\overline\cF\subset\dots\subset\cF_k\subset\dots\subset
\cF_\infty:=\jmath_*\cF_\p1.
\end{equation}

We shall use notation $\Delta$ for diagonals in $P\times P$ and
$\p1\times\p1$ through the end of this section.

\begin{Lemma}\label{Comparison}
For any $k$ and $i$ there is an isomorphism
\[(R^i p_{12,*}\cF_\p1)|_{\p1\times\p1-\Delta}\simeq
(R^i p_{12}\cF_k)|_{\p1\times\p1-\Delta}.\]
\end{Lemma}
\begin{proof}
For every $k$ we have the short exact sequence
\[
    0\to\cF_{k-1}\to\cF_k\to\iota_*(\cF_k|_{\p1\times\p1\times\Mh})\to 0,
\]
where $\iota:\p1\times\p1\times\Mh\to\p1\times\p1\times\Mb$ is
the closed embedding. Since
\[
\cF_k|_{\p1\times\p1\times\Mh}=\cF_H\otimes p_3^*\cE^{\otimes
k},
\]
Proposition~\ref{MhXxY} implies that
$Rp_{12,*}(\cF_k/\cF_{k-1})=0$ away from the diagonal, so
$Rp_{12,*}\cF_k=Rp_{12,*}\cF_{k-1}$ away from the diagonal. Now
the claim follows from the identity $R^i
p_{12,*}\cF_\p1=\limind R^i p_{12,*}\cF_k$.
\end{proof}

\begin{proof}[Proof of Proposition~\ref{AwayFromDiag}]
Consider $(x,y)\in\p1\times\p1-\Delta$. We have
\begin{equation}\label{vanishing1}
    H^i(\Mb,\overline\cF|_{(x,y)})=H^i(\cM,\cF_\p1|_{(x,y)})=0
\end{equation}
for $i\ge 2$ by Proposition~\ref{CohomologicalDimension}. Next,
$H^0(\Mb,\overline\cF|_{(x,y)})$ is finite-dimensional because
the good moduli space of $\Mb$ is projective. It follows that
$H^0(\Mb,\cF_k|_{(x,y)})=0$ for $k\ll 0$ because $\cM$ is
connected and $\cF_\p1$ is a vector bundle. Therefore
Lemma~\ref{Comparison} implies that
\begin{equation}\label{vanishing2}
    H^0(\Mb,\overline\cF|_{(x,y)})=H^0(\cM,\cF_\p1|_{(x,y)})=0.
\end{equation}
It remains to show that
$\cG:=(R^1p_{12,*}\cF_\p1)|_{\p1\times\p1-\Delta}$ vanishes. By
Lemma~\ref{Comparison},
$\cG=R^1p_{12,*}\overline\cF|_{\p1\times\p1-\Delta}$.
Moreover,~\eqref{vanishing1} and~\eqref{vanishing2} imply
that~$\cG$ is a vector bundle on $\p1\times\p1-\Delta$. Its
rank can be computed using the Euler characteristic; it equals
$-\chi(\Mb,\xi_x\otimes\xi_y^\vee)$ for any $x,y\in\p1$.
(Here~$\xi_x$ is the restriction of $\xi$ to $\{x\}\times\Mb$.)

Recall from Section~\ref{DEFMB} that the stack $\Mb$ depends on the
divisor $\divisor$ and the formal type, which we encode by
$\nu_1\in\Omega_\p1(\divisor)/\Omega_\p1$ and
$\nu_2\in\Omega^{\otimes2}_\p1(2\divisor)/\Omega^{\otimes2}_\p1(\divisor)$.
By Proposition~\ref{PrFlat}, as the parameters vary, stacks
$\Mb$ form a flat family $\Mb_{univ}$ over a Zariski open
subspace in the space of collections $(\divisor,\nu_1,\nu_2)$.
The vector bundle $\xi_x\otimes\xi_y^\vee$ makes sense in this
family. We shall use
\begin{Lemma}\label{FlatStacks}
Let $\cX\to S$ be a flat family of stacks over a scheme $S$.
Let $\cF$ be a flat sheaf on $\cX$. For $s\in S$ denote by
$\cF_s$ the fiber over $s$. Assume that there is a good moduli
space $p:\cX\to X$ such that the induced map $X\to S$ is
projective.

Then $\chi(\cF_s)$ is locally constant as a function of $s$.
\end{Lemma}
Let us apply this Lemma to $\xi_x\otimes\xi_y^\vee$. A slight
generalization of Theorems~\ref{GoodModuliSpace}
and~\ref{AmpleBundle} shows that $\Mb_{univ}$ has a good moduli
space $\oM_{univ}$, which is projective over the space of
collections $(\divisor,\nu_1,\nu_2)$. Therefore,
$\chi(\Mb,\xi_x\otimes\xi_y^\vee)$ does not depend on
$\divisor$,~$\nu_1$, and $\nu_2$. In particular, we may assume
that $\divisor=x_1+x_2+x_3+x_4$ for distinct $x_i\in\p1$ and
that $\nu_1$, $\nu_2$ are generic. Using
Lemma~\ref{Comparison}, we see that it is enough to prove that
$H^i(\cM,\xi_x\otimes\xi_y^\vee)=0$ for $x\ne y$ and all $i$ in
the case of simple $\divisor$ and generic $\nu_1$ and $\nu_2$.
This case is treated in~\cite[Theorem~2]{Arinkin}, except for a
slight difference that $SL(2)$-bundles are considered there.
However, both moduli stacks have the same good moduli space, so
the cohomology groups are the same (in fact, our moduli space
is $M\times B(\gm)$, while the moduli space in~\cite{Arinkin}
is a $\mu_2$-gerbe over~$M$).
\end{proof}

\begin{proof}[Proof of Lemma~\ref{FlatStacks}]
Note that $p_*\cF$ is flat on $X$. Indeed, if $\cG$ is a sheaf
on $X$, and $\cG^\al$ is its resolution by locally free
sheaves, we have
\[
\TOR^i(\cG,p_*\cF)=H^{-i}(\cG^\al\otimes p_*\cF)=
p_*H^{-i}(p^*\cG^\al\otimes\cF)=0.
\]
We have used the projection formula and the fact that good
moduli spaces are cohomologically affine.

By Proposition~4.7(i) of~\cite{Alper}, the restriction of $p$
to $s\in S$ is a good moduli space $p_s:\cX_s\to X_s$ so we
have
\[
    \chi(\cF_s)=\chi(p_{s,*}\cF_s)=\chi((p_*\cF)_s).
\]
(We are using a base change).

By Theorem~4.16(ix) of~\cite{Alper}, the map $X\to S$ is flat,
and we see that $\chi((p_*\cF)_s)$ is locally constant.
\end{proof}

\subsection{Proof of Proposition~\ref{SSupport}}
It is convenient to replace $\Dmod_{P,\alpha}$-modules with
modules over a certain sheaf of algebras on $\p1$. Let us make
the corresponding definitions. We identify $\Dmod_{P,\alpha}$
with a subsheaf in the pushforward of $\Dmod_{\p1-\divisor}$ to
$P$.

Let $z_i$ be a local coordinate at $x_i$. Let us lift polar
parts $\alpha_i^\pm$ to actual 1-forms on formal neighborhoods
of $x_i$; we shall denote these 1-forms by the same letters.
Consider the open embedding
$\jmath:\p1-\divisor\hookrightarrow\p1$.

Define the sheaf of algebras $\cB=\cB_{\alpha}\subset
\jmath_*\Dmod_{\p1-\divisor}$ as follows:
\begin{itemize}
\item We have $\cB|_{\p1-\divisor}=\Dmod_{\p1-\divisor}$;
\item Near $x_i\in\p1$, $\cB$ is generated by $O_\p1$,
    $z_i^{n_i}\frac{\dif}{\dif z_i}$, and
    $z_i^{-n_i}(z_i^{n_i}\frac{\dif-\alpha_i^+}{\dif
    z_i})(z_i^{n_i}\frac{\dif-\alpha_i^-}{\dif z_i})$.
\end{itemize}
Clearly, $\cB$ inherits from $\jmath_*\Dmod_{\p1-\divisor}$ the
filtration by degree of differential operators. We denote by
$\cB^{\le k}\subset\cB$ the subsheaf of operators of degree at
most $k$.

The properties of $\cB$ are summarized in the following
\begin{Proposition}\label{cB}
\noindstep\label{cB1} $\cB=\wp_*\Dmod_{P,\alpha}$.\\
\noindstep\label{cB2} Moreover, $R^1\wp_*\Dmod_{P,\alpha}=0$,
so that $\cB=R\wp_*\Dmod_{P,\alpha}$.\\
\noindstep\label{cB3} $\cB^{\le k}/\cB^{\le
k-1}=\cT_\p1^{\otimes k}(-\left\lceil\frac
k2\right\rceil\divisor)$. \lpar{}Here $\lceil\,\rceil$ is the
ceiling function.\rpar
\end{Proposition}

\begin{Remark}\label{cBRem}\stepzero
\noindsteprm The isomorphisms \eqref{cB1} and \eqref{cB3} are
naturally normalized by the condition that they become the
obvious identifications on $\p1-\divisor$.\\
\noindsteprm\label{cBRem2} Let us fix $\mu\in
H^0(\p1,\Omega_\p1^{\otimes2}(\divisor))$, $\mu\ne 0$, as
in Section~\ref{MH}. Then \eqref{cB3} can be rewritten as
\[
\cB^{\le k}/\cB^{\le k-1}=\begin{cases}
    \cT_\p1(-\divisor) &\text{if $k$ is odd,}\\
    O_\p1 &\text{if $k$ is even}.
\end{cases}
\]
\noindsteprm\label{cBRem3} Actually, $\wp:P\to\p1$ is affine
with respect to $\Dmod_{P,\alpha}$ in the sense that the
functor~$\wp_*$ is exact on $\Dmod_{P,\alpha}$-modules and
provides an equivalence between the category of
$\Dmod_{P,\alpha}$-modules and that of $\cB$-modules. We do not
use this claim, so its proof is left to the reader.
\end{Remark}

\begin{proof}[Proof of Proposition~\ref{cB}]
As we have already mentioned, the claims are obvious on
$\p1-\divisor$. Therefore, it suffices to consider the formal
neighborhood of a point $x_i$. Since we concentrate on a single
point, we drop the index $i$ to simplify the notation, so
$\alpha^\pm=\alpha^\pm_i$, $z=z_i$, and  $n=n_i$.

Let $\Dmod_K:=\C((z))\langle\frac{\dif}{\dif z}\rangle$ be the
ring of differential operators on the punctured formal
neighborhood of $x_i$. Set
\[
    \delta:=z^n\frac\dif{\dif z},\quad
    B:=z^{-n}\left(z^n\frac{\dif-\alpha^+}{\dif z}\right)
    \left(z^n\frac{\dif-\alpha^-}{\dif z}\right)\in\Dmod_K
\]
and
\[
    \cB_O:=\C[[z]]\left\langle\delta,B\right\rangle\subset\Dmod_K,\qquad
    \Dmod_O^\pm:=\C[[z]]\left\langle\frac{\dif-\alpha^\pm}{\dif z}\right\rangle\subset\Dmod_K.
\]
Then the restriction of $\cB$ to the formal neighborhood of
$x_i$ is $\cB_O$, the restriction of $\Dmod_{P,\alpha}$ to the
formal neighborhoods of $x_i^\pm$ is $\Dmod_O^\pm$, and the
restriction of\nopagebreak{} $R^0\wp_*\Dmod_{P,\alpha}$ (resp.
$R^1\wp_*\Dmod_{P,\alpha}$) to the formal neighborhood of $x_i$
equals $\Dmod_O^+\cap\Dmod_O^-$ (resp.
$\Dmod_K/(\Dmod_O^++\Dmod_O^-)$). The proposition thus reduces
to the following statements:
\begin{enumerate}
\item\label{Odin} $\cB_O=\Dmod_O^+\cap\Dmod_O^-$,
\item\label{Dva} $\Dmod_K=\Dmod_O^++\Dmod_O^-$,
\item\label{Tri} The set $\{1,\delta,B,B\delta ,B^2,
    B^2\delta,\dots\}$ is a basis of $\cB_O$ as of a
    $\C[[z]]$-module. Note that the symbol of $\delta$
    (resp.\ the symbol of $B$) is a section of
    $\cT_\p1(-\divisor)$ (resp. of
    $\cT_\p1^{\otimes2}(-\divisor)$).
\end{enumerate}

Set $F:=\C[[z]]\langle\delta\rangle\subset\Dmod_K$, and
introduce the filtration
\[
    \dots\subset z F\subset F\subset z^{-1}F\subset\dots\subset\Dmod_K.
\]
For an element $C\in\Dmod_K$ denote by $\bar C$ its image in
$\gr\Dmod_K$.

\begin{Lemma}\label{lm:Filtration}\stepzero
\noindstep\label{Gr1} This filtration is exhaustive, separated,
and compatible with the ring structure.

\noindstep\label{Gr2} If $n>1$, then the associated graded ring
is isomorphic to $\C[\bar z,\bar z^{-1},\bar\delta]$, that is,
to the ring of functions on $\A1\times(\A1-0)$.

\noindstep\label{Gr3} For $n=1$ the associated graded ring is
isomorphic to $\C[\bar z,\bar
z^{-1}]\langle\bar\delta\rangle/(\bar\delta\bar z-\bar
z\bar\delta-\bar z)$, that is, to the ring of differential
operators on $\A1-0$.
\end{Lemma}
\begin{proof}
(\ref{Gr1}) Note first that every element of $C\in\Dmod_K$ can
be written uniquely as
\begin{equation}\label{present}
    C=\sum_{l\ge0} f_l(z)\delta^l,
\end{equation}
where $f_l(z)\in\C((z))$. It follows from commutation relation
\begin{equation}\label{commute}
[\delta,z^k]=kz^{k+n-1}
\end{equation}
that $C\in F$ if and only if for all $l$ we have
$f_l(z)\in\C[[z]]$. Thus $C\in z^{-k}F$ if and only if for all
$l$ we have $z^kf_l(z)\in\C[[z]]$. Hence the filtration is
exhaustive and separated.

It follows from commutation relation~(\ref{commute}) by
induction on $l$ that $\delta^lz^k\in z^kF$. Now it is easy to
see that the filtration is compatible with the ring structure.

(\ref{Gr2}) It follows from (\ref{commute}) that $\bar z$ and
$\bar\delta$ commute in $\gr\Dmod_k$ if $n>1$. Thus we get a
homomorphism
\[
    \C[\bar z,\bar z^{-1},\bar\delta]\to\gr\Dmod_K.
\]
Using presentation~(\ref{present}), we see that it is
bijective.

The proof of~(\ref{Gr3}) is similar to that of~(\ref{Gr2}).
\end{proof}

Denote by $a^\pm$ the leading coefficient of $\alpha^\pm=a^\pm
z^{-n}\dif z+\dots$, and define the polynomials $q^\pm_l(t)$
for a non-negative integer $l$ by
\[
    q_l^\pm(t):=\begin{cases}
        (t-a^\pm)^l &\text{ if }n>1,\\
        \prod_{i=0}^{l-1}(t-a^\pm-i) &\text{ if }n=1.
    \end{cases}
\]
\begin{Lemma}\label{Image}
\stepzero\noindstep\label{Image1} The image of
$\left(\frac{\dif-\alpha^\pm}{\dif z}\right)^l$
in $\gr\Dmod_K$ is $\bar z^{-nl}q_l^\pm(\bar\delta)$.\\
\noindstep\label{Image2} The image of $B^l$ in $\gr\Dmod_K$ is
$\bar z^{-nl}q_l^+(\bar\delta)q_l^-(\bar\delta)$.
\end{Lemma}
\begin{proof}
(\ref{Image1}) The image of $\frac{\dif-\alpha^\pm}{\dif z}$ in
$\gr\Dmod_K$ is $\bar z^{-n}(\bar\delta-a^\pm)$. If $n>1$, then
the statement follows from commutativity of
$\gr\Dmod_K$.

If $n=1$, then we have to move all copies of $\bar z^{-1}$ to
the left in $(\bar z^{-1}(\bar\delta-a^\pm))^l$. Now the
statement follows from the relation
\begin{equation}\label{commute2}
    (\bar\delta-a)\bar z^{-1}=\bar z^{-1}(\bar\delta-a-1),\quad a\in\C.
\end{equation}

(\ref{Image2}) We have $\bar B=\bar
z^{-n}(\bar\delta-a^+)(\bar\delta-a^-)$. Now the case $n>1$ is
obvious, the case $n=1$ again follows from~(\ref{commute2}).
\end{proof}

By Lemma~\ref{lm:Filtration} any element of $\gr_{-k}\Dmod_K$
can be uniquely written as
\[
    \bar z^{-k}p(\bar\delta),\quad p(\bar\delta)\in\C[\bar\delta].
\]
Denote by $\gr\Dmod_O^\pm$ the set of images of all elements of
$\Dmod_O^\pm$ in $\gr\Dmod_K$. Define $\gr\cB_O$ similarly. Fix
$k\in\Z$ and set $l:=\lceil\frac kn\rceil$.
\begin{Lemma}
\begin{align*}
&(a)\qquad\bar z^{-k}p(\bar\delta)\in\gr\Dmod_O^\pm &&\text{ if and only if }
k\le 0\text{ or }q_l^\pm(t)|p(t);\\
&(b)\qquad\bar z^{-k}p(\bar\delta)\in\gr\cB_O &&\text{ if and only if }
k\le 0\text{ or }q_l^+(t)q_l^-(t)|p(t).
\end{align*}
\end{Lemma}
\begin{proof}
(a) Consider any element $C\in\Dmod_O^+$, $C\ne0$. It is easy
to see that it can be uniquely written as
\[
    \sum_{i,j\ge0}f_{ij}z^j\left(\frac{\dif-\alpha^+}{\dif z}\right)^i
\]
with $f_{ij}\in\C$. Let $k$ be the maximum value of the
function $(i,j)\mapsto ni-j$ on the set
$\{(i,j)|\,f_{ij}\ne0\}$. Then Lemma~\ref{Image}(\ref{Image1})
shows that
\[
    C\in\sum_{ni-j=k}f_{ij}z^{-k}q_i^+(\delta)+z^{1-k}F.
\]
Since elements $\bar z^{-k}q_i^+(\bar\delta)$ form a basis in
$\gr_{-k}\Dmod_K$, we see that $C\notin z^{1-k}F$ and
\[
    \bar C=\sum_{ni-j=k}f_{ij}\bar z^{-k}q_i^+(\bar\delta).
\]
Since $j\ge0$, we see that $i\ge\frac kn$, so $i\ge l$. Thus if
$k>0$, then $\bar C=\bar z^{-k}p(\bar\delta)$, where $p$ is
divisible by $q_l^+$.

Conversely, given a polynomial $p$ divisible by $q_l^+$ (or any
polynomial if $k\le0$), we can write $p=\sum_{i\ge l}f_iq_i^+$
with $f_i\in\C$. Set
\[
    C=\sum_if_i z^{ni-k}\left(\frac{\dif-\alpha^+}{\dif z}\right)^i.
\]
Then $C\in\Dmod_O^+$ and $\bar C=\bar z^{-k}p(\bar\delta)$. The
case of $\Dmod_O^-$ is completely similar.

(b) Consider $C\in\cB_O$ with $\bar C=\bar
z^{-k}p(\bar\delta)$. It is easy to see that
$\cB_O\subset\Dmod_O^+\cap\Dmod_O^-$. Thus it follows from part
(a) that $q_l^\pm(t)$ divides $p(t)$. Thus $q_l^+(t)q_l^-(t)$
divides $p(t)$, since $q_l^-(t)$ and $q_l^+(t)$ are coprime.
Finally, assume that $q_l^+(t)q_l^-(t)$ divides $p(t)$, we can
write
\[
    p(t)=\sum_{i\ge l}f_iq_i^+(t)q_i^-(t)+
    \sum_{i\ge l}g_itq_i^+(t)q_i^-(t),\qquad f_i,g_i\in\C.
\]
Set
\[
    C=\sum_i f_iz^{ni-k}B^i+\sum_i g_i z^{ni-k}B^i\delta.
\]
Clearly, $\bar C=\bar z^{-k}p(\bar\delta)$.
\end{proof}

We now see that the identities \eqref{Odin}--\eqref{Dva} hold
in the associated graded ring of $\Dmod_K$, and hence also in
$\Dmod_K$ itself. The proof of part~(b) of the lemma shows that
every element of $\gr\cB_O\cap\gr_{-k}\Dmod_K$ can be uniquely
written as
\[
    \sum_{i\ge k/n} f_i\bar z^{ni-k}\bar B^i+\sum_{i\ge k/n}
    g_i \bar z^{ni-k}\bar B^i\bar\delta
\]
with $f_i,g_i\in\C$, and~\eqref{Tri} follows. The proof of
Proposition~\ref{cB} is complete.
\end{proof}

\begin{proof}[Proof of Proposition~\ref{SSupport}]
Our first goal is to reformulate the proposition as a statement
about the cokernel of a map between sheaves on $\p1\times\p1$.

Note first of all that any
$p_1^\al\Dmod_{P,\alpha}\circledast
p_2^\al\Dmod_{P,-\alpha}$-module $\cG$ on $P\times P$ is,
in particular, a $p_1^{-1}\Dmod_{P,\alpha}$-module.
Therefore, $(\wp\times\wp)_*\cG$ has a natural structure of
a $p_1^{-1}\cB$-module coming from the isomorphism of
Proposition~\ref{cB}\eqref{cB1}. (There is also a commuting
structure of a $p_2^{-1}\cB_{-\alpha}$-module that we do
not use.)

Thus
\begin{equation}\label{pph1phi}
    (\wp\times\wp)_*H^1(\phi):(\wp\times\wp)_*\delta_\Delta\to
    (\wp\times\wp)_*R^1p_{12,*}\cF_P
\end{equation}
is a map of $p_1^{-1}\cB$-modules.

Consider $\delta_\Delta$ as a
$p_1^{-1}\Dmod_{P,\alpha}$-module. It is isomorphic to
$\iota_{\Delta,*}(\Dmod_{P,\alpha}\otimes_{O_P}\cT_P$), where
$\iota_{\Delta,*}$ is the $O$-module pushforward. By the
projection formula, Proposition~\ref{cB} gives an isomorphism
in the derived category of $p_1^{-1}\cB$-modules
\begin{equation}\label{rpp}
    R(\wp\times\wp)_*\delta_\Delta=\iota_{\Delta,*}(\cB\otimes_{O_\p1}\cT_\p1).
\end{equation}
Using this and~(\ref{dirimage}), we re-write~(\ref{pph1phi}) as
\begin{equation}\label{pph2phi}
    (\wp\times\wp)_*H^1(\phi):\iota_{\Delta,*}(\cB\otimes_{O_\p1}\cT_\p1)\to
    R^1p_{12,*}\cF_\p1.
\end{equation}
As was explained in the proof of Theorem~\ref{Theorem3},
$H^1(\phi)$ is injective.  Also,
$R^1(\wp\times\wp)_*\delta_\Delta=0$ by~(\ref{rpp}). We now see
that
\[
    (\wp\times\wp)_*\coker(H^1(\phi))=\coker((\wp\times\wp)_*H^1(\phi)).
\]
Thus it remains to prove that the cokernel of~(\ref{pph2phi})
is coherent. Note that~(\ref{pph2phi}) is an injective maps
between $p_1^{-1}\cB$-modules.

Now recall that $\cF_\p1$ naturally extends to a vector bundle
$\overline\cF:=p_{13}^*\xi\otimes p_{23}^*\xi^\vee$ on
$\p1\times\p1\times\Mb$, which provides the filtration
\eqref{filtration}. By Proposition~\ref{MhXxY} we obtain an
induced filtration
\begin{equation}\label{R1filtration}
R^1p_{12,*}\cF_0\subset\dots\subset
R^1p_{12,*}\cF_k\subset\dots\subset
R^1p_{12,*}\cF_\infty=R^1p_{12,*}\cF_\p1
\end{equation}
of $R^1p_{12,*}\cF_\p1$ by coherent sheaves on $\p1\times\p1$.

\begin{Lemma}
There is $l\in\Z$ such that for $k\gg 0$ the image of
$\iota_{\Delta,*}(\cB^{\le k}\otimes_{O_\p1}\cT_\p1)$
under~(\ref{pph2phi}) is contained in $R^1p_{12,*}\cF_{k+l}$
and such that the induced map
\begin{equation}\label{isom}
\iota_{\Delta,*}((\cB^{\le k+1}/\cB^{\le k})\otimes_{O_\p1}\cT_\p1)\to
R^1p_{12,*}\cF_{k+l+1}/R^1p_{12,*}\cF_{k+l}.
\end{equation}
is an isomorphism.
\end{Lemma}
\begin{proof}
By construction, the filtration \eqref{filtration} agrees with
the filtration on $\cB$, so that $(p_1^{-1}\cB^{\le
l})\cF_k\subset\cF_{k+l}$. Therefore, the filtration
\eqref{R1filtration} also agrees with the filtration on $\cB$.
 Using
Remark~\ref{cBRem}(\ref{cBRem2}) and Proposition~\ref{MhXxY},
we see that for $k\gg 0$ {\small
\begin{equation}\label{assgraded}
\iota_{\Delta,*}((\cB^{\le k}/\cB^{\le k-1})\otimes_{O_\p1}\cT_\p1)\simeq
R^1p_{12,*}\cF_k/R^1p_{12,*}\cF_{k-1}\simeq
\begin{cases}
     \iota_{\Delta,*}\cT^{\otimes 2}_\p1(-\divisor) &\text{if $k$ is odd,}\\
     \iota_{\Delta,*}\cT_\p1 &\text{if $k$ is even}.
\end{cases}
\end{equation}}

For each $k$, let $l(k)$ be the smallest index such that the
image of $\iota_{\Delta,*}(\cB^{\le k}\otimes_{O_\p1}\cT_\p1)$
is contained in $R^1p_{12,*}\cF_{k+l(k)}.$ Since the filtration
\eqref{R1filtration} agrees with the filtration on $\cB$, we
see that $l(k+1)\le l(k)$ for all $k$. Also, injectivity of
\eqref{pph2phi} implies that
\[0\le \rk\iota_\Delta^*(R^1p_{12,*}\cF_{k+l(k)})-\rk(\cB^{\le k}\otimes_{O_\p1}\cT_\p1)=
(k+l(k)+\rk\iota_\Delta^*(R^1p_{12,*}\cF_0))-k,\] and therefore
$l(k)\ge -\rk\iota_\Delta^*(R^1p_{12,*}\cF_0)$. Thus $l(k)$
stabilizes as $k\to\infty$; set $l:=\lim l(k)$.

By the choice of $l$, the map~\eqref{isom} is non-zero for
$k\gg 0$. Note that such non-zero map does not exist if $k$ and
$l$ are odd. Therefore, $l$ must be even. We now see that for
$k\gg0$, the map~\eqref{isom} is a non-zero morphism between
isomorphic line bundles on $\Delta$. This implies~\eqref{isom}
is an isomorphism.
\end{proof}

It follows from the lemma that, if $k$ is large enough,
$R^1p_{12,*}\cF_k$ maps surjectively onto $\coker(\wp\times\wp)_*H^1(\phi)$.
This completes the proof of the proposition and of
Theorem~\ref{Theorem3}.
\end{proof}

\section{Second orthogonality relation.}\label{LYSENKO}
In this section we prove Theorem~\ref{ThLys}. The proof is
similar to~\cite{Arinkin} but we want to give some details.

We need to calculate $Rp_{12,*}\DR(\cF_\cM)$, where
$p_{12}:\cM\times\cM\times P\to\cM\times\cM$. Our first goal is
to reduce the problem to a calculation on
$\cM\times\cM\times\p1$. Recall that $\xi$ is the universal
bundle on $\cM\times\p1$, set $\xi_{12}:=\HOM(p_{23}^*\xi,
p_{13}^*\xi)$. We have a connection along $\p1$
\[
\ad\nabla:\xi_{12}\to\xi_{12}\otimes p_3^*\Omega_\p1(\divisor).
\]
Its polar part is a well-defined $O$-linear map
\[
    \xi_{12}|_{\cM\times\cM\times\divisor}\to
    (\xi_{12}\otimes p_3^*\Omega_\p1(\divisor))|_{\cM\times\cM\times\divisor}.
\]
Let $\eta_{12}$ be the image of this map. Denote by $\tilde
\xi_{12}$ the modification of $\xi_{12}\otimes
p_3^*\Omega_\p1$ whose sheaf of sections is
\[
    \{s\in\xi_{12}\otimes p_3^*\Omega_\p1(\divisor)|_{\cM\times\cM\times\divisor}
    :s|_{\cM\times\cM\times\divisor}\in\eta_{12}\}.
\]
As in~\cite[Lemmas~12, 13]{Arinkin} one proves that
\[
   Rp_{12,*}\DR(\cF_\cM)=
   Rp_{12,*}(\xi_{12}\xrightarrow{\ad\nabla}\tilde\xi_{12}).
\]
The next step is to calculate the restriction of the above to a
fiber over a point. So consider a closed point
$x\in\cM\times\cM$ and let $\cF^\al$ be the restriction of
$\xi_{12}\xrightarrow{\ad\nabla}\tilde\xi_{12}$ to $x$.

\begin{Proposition}\label{PpLys}
\stepzero\noindstep\label{PpLys1} If $x\notin\diag(\cM)$, then
$\HH^i(\cF^\al)=0$ for any $i$;\\
\noindstep\label{PpLys2} If $x\in\diag(\cM)$, then
\[
 \dim\HH^i(\cF^\al)=\begin{cases}
   1\text{ if }i=0,2,\\ 2\mbox{ if }i=1,\\
   0\text{ otherwise.}
 \end{cases}
\]
\noindstep\label{PpLys3} Suppose
$x=((L,\nabla),(L,\nabla))\in\cM\times\cM$. Consider the map of
complexes
\[
  (O_\p1\xrightarrow{\dif}\Omega_\p1)\hookrightarrow\cF^\al
\]
induced by $O_\p1\to\xi_{12}|_x:f\mapsto f\id_L$. Then the
induced map
\[
    H^i_{DR}(\p1,\C):=
    \HH^i(\p1,O_\p1\xrightarrow{\dif}\Omega_\p1)\to\HH^i(\cF^\al)
\]
is an isomorphism for $i=0,2$.
\end{Proposition}
The proof is analogous to the proof
of~\cite[Proposition~10]{Arinkin}: one uses irreducibility,
duality, and Euler characteristic.

As in~\cite[Lemma~14]{Arinkin} the duality gives the following
\begin{Lemma}
\label{RelLys} Let $S$ be a locally Noetherian scheme,
$\iota:S\to\cM\times\cM$. Set
\[
    \cF_{(S)}:=Rp_{1,*}((\iota\times\id_\p1)^*
    (\xi_{12}\xrightarrow{\ad\nabla}\tilde\xi_{12}))
\]
\lpar{}here $p_1:S\times\p1\to S$\rpar. Then
$\HOM(H^2(\cF_{(S)}),O_S)$ is isomorphic to a subsheaf of
$H^0(\cF_{(S)})$.
\end{Lemma}

Next, $\diag:\cM\to\cM\times\cM$ is a $\gm$-torsor over
$\diag(\cM)$ (cf. Remark~\ref{RemDiag}). Denote by $\Hom$ the
corresponding line bundle. Note that the fiber of $\Hom$ over
$((L_1,\nabla_1),(L_2,\nabla_2))$ is
$\{A\in\Hom_{O_\p1}(L_1,L_2): A\nabla_1=\nabla_2A\}$. Now the
following corollary of Proposition~\ref{PpLys} is obvious.

\begin{Corollary}\label{CoLys}
\stepzero\noindstep\label{CoLys1}
$Rp_{12,*}(\xi_{12}\xrightarrow{\ad\nabla}\tilde\xi_{12})$
vanishes if restricted to $\cM\times\cM-\diag(\cM)$.

\noindstep\label{CoLys2} The map
\[
p_{12}^*\Hom\otimes p_3^*(O_\p1\xrightarrow{\dif}\Omega_\p1)
\to(\xi_{12}\xrightarrow{\ad\nabla}\tilde\xi_{12})|_{\diag(\cM)}
\]
induces an isomorphism
\[
\Hom=\Hom\otimes\HH^2(\p1,(O_\p1\xrightarrow{\dif}\Omega_\p1))\to
R^2p_{12,*}\left((\xi_{12}\xrightarrow{\ad\nabla}\tilde\xi_{12})|_{\diag(\cM)}\right).
\]
\end{Corollary}

Let us use the following observation
(cf.~\cite[Lemma~~15]{Arinkin} and~\cite[Lemma in
Section~13]{MumfordAbelian}).

\begin{Lemma}\label{LmLys}
Let $Z$ be a locally Noetherian scheme, $Y\subset Z$ a closed
subscheme that is locally a complete intersection of pure
codimension $n$. Denote by $\iota:Y\hookrightarrow Z$ the
natural embedding.

\indent\stepzero\noindstep\label{LmLys1} Let $\cF$ be a
quasi-coherent sheaf on $Z$ such that $\cF|_{Z-Y}=0$,
$L_n\iota^*\cF=0$. Then $\cF=0$.

\indent\noindstep\label{LmLys2} Let
$\cF^\al=(\cF^0\to\cF^1\to\dots)$ be a complex of flat
quasi-coherent sheaves on $Z$ such that $H^i(\cF^\al)|_{Z-Y}=0$
for all $i<n$. Then $H^i(\cF^\al)=0$ for $i<n$.
\end{Lemma}

\begin{proof}[Proof of Theorem~\ref{ThLys}]
Clearly, $\diag(\cM)=M\times B(\gm\times\gm)$ is a complete
intersection in $\cM\times\cM=M\times M\times B(\gm\times\gm)$.
Thus Lemma~\ref{LmLys}(\ref{LmLys2}) and
Corollary~\ref{CoLys}(\ref{CoLys1}) imply that
$R^ip_{12,*}(\xi_{12}\xrightarrow{\ad\nabla}\tilde\xi_{12})=0$
for $i\ne2$. Set
$\cF^{(2)}:=R^2p_{12,*}(\xi_{12}\xrightarrow{\ad\nabla}\tilde\xi_{12})$.
Corollary~\ref{CoLys}(\ref{CoLys2}) implies that
$\Hom=\cF^{(2)}|_{\diag(\cM)}$. It is easy to see that $\Hom$,
viewed as a sheaf on $\cM\times\cM$, is equal to
$(\diag_*O_\cM)^\psi$, where $\psi$ is the character of
$\gm\times\gm$ given by $(t_1,t_2)\mapsto t_1/t_2$ (because a
1-dimensional vector space $E$ can be identified with weight
$-1$ functions on $E-\{0\}$).

To complete the proof, it remains to check that $\cF^{(2)}$ is
concentrated (scheme-theoretically) on $\diag(\cM)$. Assume for
a contradiction that it is not the case. Note that $\cF^{(2)}$
is coherent and concentrated set-theoretically on $\diag(\cM)$.

\begin{Lemma}\label{lm:Nilpotents}
Let $Z$ be a locally Noetherian scheme, $Y\subset Z$ be a
closed subscheme. Let $\cG$ be a coherent sheaf on $Z$
concentrated set-theoretically but not scheme-theoretically on
$Y$. Then there is a local Artinian scheme $S$ and an $S$-point
of $Z$ such that $S^{red}$ factors through $Y$ and such that
the restriction of $\cG$ to $S$ is not concentrated on the
scheme-theoretic preimage of $Y$.
\end{Lemma}

We see that there is an $S$-point of $\cM\times\cM$ such that
the restriction of $\cF^{(2)}$ to this point is not
concentrated on the preimage $S'$ of the diagonal. Let
$\cF_{(S)}$ be as in Lemma~\ref{RelLys}; using base change we
see that $H^2(\cF_{(S)})$ is not concentrated on $S'$. The
duality for Artinian rings shows that
$\HOM(H^2(\cF_{(S)}),O_S)$ is not concentrated on $S'$ either.
But then Lemma~\ref{RelLys} gives a contradiction, since
$H^0(\cF_{(S)})$ is easily seen to be concentrated on $S'$.
\end{proof}

\begin{proof}[Proof of Lemma~\ref{lm:Nilpotents}]
We can assume that $Z=\spec A$, $Y=\spec A/\mathfrak a$, $\cG$
corresponds to an $A$-module $M$; by assumption $\mathfrak
aM\ne0$ but $\mathfrak a^nM$=0 for $n$ big enough. Consider a
maximal ideal $\mathfrak m$ such that $(\mathfrak
aM)_{\mathfrak m}\ne0$. It follows that $\mathfrak
m\supset\mathfrak a$. By Nakayama's Lemma $\cap_n\mathfrak
m^nM_{\mathfrak m}=0$ and we can choose~$n$ such that
$\mathfrak a M\not\subset\mathfrak m^nM$. We can take
$S=\spec(A/\mathfrak m^n)$.
\end{proof}

\section{Relation to the Langlands correspondence}
\label{ProofOfLanglands} In this section we prove
Theorem~\ref{Langlands}. Let us present the main steps. Recall
that $\divisor=\sum n_ix_i$. Set
\[
\beta_i^+:=\alpha_i^++\frac{n_i\lambda}2\frac{\dif z_i}{z_i},\qquad
\beta_i^-:=\alpha_i^--\frac{n_i\lambda}2\frac{\dif z_i}{z_i},
\]
where $\lambda:=\sum_i\res\alpha_i^-$, $z_i$ is a local
parameter at $x_i$. Note that the polar parts $\beta_i^\pm$ do
not depend on the choice of $z_i$. Denote now the moduli space
$\cM$ defined in Section~\ref{MODST} by $\cM_\alpha$ to make the
choice of formal types explicit. For a sheaf $\cR$ of rings we
denote by $\cR$-mod the category of left $\cR$-modules, by
$\cD^b(\cR)$ the bounded derived category of $\cR$-mod. We
shall prove first that
\[
    {\Dmod_{\overline\Bun(-1),\alpha}\text{-mod}}=
    \Dmod_{\Bun(-1),\alpha}\text{-mod}
    \simeq\Dmod_{P,\beta}\text{-mod}.
\]
It remains to
prove the following equivalences
\[
  \cD^b(\Dmod_{P,\beta})\simeq \cD^b(\cM_\beta)^{(-1)}
  \simeq\cD^b(\cM_\alpha)^{(-1)}.
\]
Note that if $\alpha_i^\pm$ satisfy the conditions
of Section~\ref{MODST}, then $\beta_i^\pm$ satisfy these conditions
as well ((\ref{AlphaIII}) and~(\ref{AlphaIV}) can be checked
case by case). Thus the first equivalence follows from
Theorem~\ref{MainTh}. For the last equivalence we shall prove
that $\cM_\alpha\simeq\cM_\beta$.

It is well known
that the definition of the derived category of $\Dmod$-modules
on a stack requires some caution. In this paper, we ignore the
difficulty and use the naive definition: the derived category
of $\Dmod_{\overline\Bun(-1),\alpha}$-modules is simply the
derived category of the abelian category of
$\Dmod_{\overline\Bun(-1),\alpha}$-modules.

\subsection{Twisted $\Dmod$-modules on algebraic
stacks}\label{TDOEXISTENCE} Let us summarize the properties of
modules over TDO rings on algebraic stacks. We make no attempt
to work in most general settings, and consider only smooth
stacks, and only twists induced by torsors over an algebraic
group. This case is enough for our purposes.

Let $G$ be an algebraic group with Lie algebra $\Lie(G)$. Fix a
$G$-invariant functional $\theta:\Lie(G)\to\C$. First, consider
twisted differential operators on a variety.

Let $X$ be a smooth variety and let $p:T\to X$ be a $G$-torsor
on $X$. These data determine a TDO ring $\Dmod_{X,T,\theta}$ on
$X$, which is obtained by non-commutative reduction of the
sheaf of differential operators $\Dmod_T$ on $T$.

Namely, every $\xi\in\Lie(G)$ gives a first order differential
operator $a(\xi)-\theta(\xi)\in\Dmod_T$, where the vector field
$a(\xi)$ on $T$ is the action of $\xi$. Let $I$ be the ideal in
$p_*\Dmod_T$ generated by these differential operators. It is
easy to see that this ideal is $G$-invariant, and we set
$\Dmod_{X,T,\theta}:=(p_*\Dmod_T/I)^G$.

The category of quasi-coherent $\Dmod_{X,T,\theta}$-modules can
be described using a twisted strong equivariance condition. Let
$M$ be a $\Dmod_T$-module equipped with a weak $G$-equivariant
structure (that is, $M$ is $G$-equivariant as a quasi-coherent
sheaf, and the structure of a $\Dmod_T$-module is
$G$-equivariant). We say that $M$ is \emph{strongly equivariant
with twist $\theta$} if the action of $\xi\in\Lie(G)$ on $M$
induced by the $G$-equivariant structure is given by
$a(\xi)-\theta(\xi)$.

\begin{Remark}
The sheaves of twisted differential operators have been
introduced in \cite{BeilinsonBernstein}. The correspondence
between $\Dmod_{X,T,\theta}$-modules and twisted strongly
equivariant modules is a particular case of the formalism of
Harish-Chandra algebras from \cite[Section~1.8]{BeilinsonBernstein}.
\end{Remark}

Let now $\cX$ be an algebraic stack, and let $T\to\cX$ be a
$G$-torsor on $\cX$. Every smooth morphism $\alpha:X\to\cX$
from a variety $X$ induces a $G$-torsor $\alpha^*T$ on $X$, and
we obtain the TDO ring $\Dmod_{X,\alpha^*T,\theta}$. Such TDO
rings form a $\Dmod$-algebra on $\cX$ in the sense of
\cite{BeilinsonBernstein}. We denote this $\Dmod$-algebra by
$\Dmod_{\cX,T,\theta}$. Note that $\Dmod_{\cX,T,\theta}$ is not
a sheaf of algebras on $\cX$.

By definition, a $\Dmod_{\cX,T,\theta}$-module $M$ is given by
specifying a $\Dmod_{X,\alpha^*T,\theta}$-module~$M_\alpha$ for
every smooth morphism $\alpha:X\to\cX$ and an isomorphism of
$\Dmod_{Y,(\alpha\circ f)^*T,\theta}$-modules
$f^*M_\alpha\simeq M_{\alpha\circ f}$ for every smooth map
$f:Y\to X$ of algebraic varieties; the isomorphisms must be
compatible with composition of morphisms $f$. Note in
particular that $M$ is a quasi-coherent sheaf on $\cX$.

\begin{Example}\label{Ex:ClassStack}
Let $\cX:=B(H)$ be the classifying stack of an algebraic group
$H$. Set $X=\spec\C$. The natural map $\alpha:\spec\C\to\cX$ is
an $H$-torsor (and, in particular, a presentation). For any
$G$-torsor $T$ on $\cX$, the pullback $\alpha^*T$ is isomorphic
to the trivial torsor $G\to\spec\C$. Fix a trivialization
$\alpha^*T\simeq G$. The group $H$ acts on $\alpha^*T=G$; this
is the right action for a homomorphism $\psi:H\to G$. In other
words,~$T$ is the descent of $G\to\spec\C$, and $\psi$ provides
the descent datum.

Let $M$ be a $\Dmod_{\cX,T,\theta}$-module. It is easy to see
that the TDO ring $\Dmod_{\spec\C,G,\theta}$ is just the field
of complex numbers, so the $\Dmod_{\spec\C,G,\theta}$-module
$\alpha^*M$ is a vector space $V$. Let us view $\alpha^*M$ as a
strongly $G$-equivariant $\Dmod_{\alpha^*T}$-module with twist
$\theta$. It corresponds to the free $O_G$-module $V\otimes_\C
O_G$ with the obvious $G$-equivariant structure. The action of
$\Dmod_{\alpha^*T}$ is uniquely determined by the twisted
strong equivariance condition. On the other hand, $\alpha^*M$
also carries a structure of a strongly $H$-equivariant
$\Dmod$-module; this structure is essentially the descent data
for $M$. If $V\ne0$, then such structure is provided by a
scalar representation of $H$ on $V$ whose derivative is
$-\theta\circ\dif\psi$.

In particular, suppose that the character
$\theta\circ\dif\psi:\Lie(H)\to\C$ does not integrate to a
representation $H_0\to\gm$, where $H_0\subset H$ is the
identity component. Then $V=0$ and therefore the only
$\Dmod_{\cX,T,\theta}$-module is the zero module.
\end{Example}

\subsection{Step 1:
$\Dmod_{\overline\Bun(-1),\alpha}-
\mathrm{mod}=\Dmod_{\Bun(-1),\alpha}-\mathrm{mod}$} In this
section we shall prove

\begin{Proposition}
Assume that $(L,\eta)\in\overline\Bun(-1)$ does not correspond
to a connection $(L,\nabla)\in\cM$ in the sense
of Section~\ref{AFFSTR}. Then the restriction of any
$\Dmod_{\overline\Bun(-1),\alpha}$-module $\xi$ to $(L,\eta)$
is zero.
\end{Proposition}
This proposition and Proposition~\ref{Affine}\eqref{Aff1} imply
that, for a point \[(L,\eta)\in \overline\Bun(-1)-\Bun(-1),\]
the restriction of every $\Dmod_{\overline\Bun(-1),\alpha}$-module to
$(L,\eta)$ is zero. This is valid for
not necessarily closed points, and the step follows.

\begin{proof}[Proof of proposition]
The proof is based on Corollary~\ref{ConnExist} and
Example~\ref{Ex:ClassStack}.

Consider Example~\ref{Ex:ClassStack} with $H=\Aut(L,\eta)$.
Recall from Section~\ref{LANGLANDS} that the twist is given by the
torsor
\[
    T=\eta_{univ}\times_{\overline\Bun(-1)}\eta'_{univ}
\]
over $G=\C[\divisor]^\times\times\C[\divisor]^\times$ and the
character $\theta=(\alpha^+,\alpha^-)$. One easily checks that
$\psi:\Aut(L,\eta)\to\C[\divisor]^\times\times\C[\divisor]^\times$
is given by the action of automorphisms on $\eta$ and
$(L|_\divisor)/\eta$. Thus, in the notation of
Corollary~\ref{ConnExist}, $\dif\psi:A\mapsto(A_+,A_-)$ and
\begin{equation}\label{chipsi}
    \theta\circ\dif\psi:A\mapsto\res(A_+\alpha_+)+\res(A_-\alpha_-).
\end{equation}
Now assume that $(L,\nabla)$ does not correspond to any
connection; our goal is to prove that $\theta\circ\dif\psi$
does not integrate to a character of the identity component of
$\Aut(L,\eta)$. It follows from Corollary~\ref{ConnExist} that
there is $A\in\End(L,\eta)$ such that
\[
    \res(A_+\alpha_+)+\res(A_-\alpha_-)+\langle A,b(L)\rangle\ne0.
\]
It is enough to consider two cases: $A$ is nilpotent, and $A$
is semisimple. In the first case it follows from~\eqref{chipsi}
and~\eqref{atiyah} that $\theta\circ\dif\psi(A)\ne0$ and
$\theta\circ\dif\psi$ cannot integrate to a character
$H_0\to\gm$, so we are done.

Let $A$ be semisimple. It follows from~\eqref{atiyah} and
condition~\eqref{AlphaII} of Section~\ref{MODST} that
\[
    \res((\id_L)_+\alpha_+)+\res((\id_L)_-\alpha_-)+\langle\id_L,b(L)\rangle=0.
\]
Thus $A$ is not scalar. Then $(L,\eta)$ decomposes with respect
to the eigenvalues of~$A$ as $L=L_1\oplus L_2$ (and for every
$i$ we have $\eta|_{n_ix_i}=L_1|_{n_ix_i}$ or
$\eta|_{n_ix_i}=L_2|_{n_ix_i}$). Let~$A'$ be the endomorphism
of $(L,\eta)$ that is zero on $L_1$ and the identity on $L_2$.
We see that
\[
    \res(A'_+\alpha_+)+\res(A'_-\alpha_-)=\sum_i\res\alpha_i^\pm\notin\Z
\]
by condition~\eqref{AlphaIII} of Section~\ref{MODST}. Again,
$\theta\circ\dif\psi$ does not integrate and we are done.
\end{proof}

\subsection{Step 2:
$\Dmod_{\Bun(-1),\alpha}-\mathrm{mod}\simeq
\Dmod_{P,\beta}-\mathrm{mod}$} Recall that $\Bun(-1)=P/\gm$,
where $\gm$ acts trivially.

Let $\pi:P\to\Bun(-1)$ be the projection. It follows from the
definition of a (strongly) equivariant $\Dmod$-module that
$\Dmod_{\Bun(-1),\alpha}-\mathrm{mod}\simeq
\pi^\al\Dmod_{\Bun(-1),\alpha}-\mathrm{mod}$. So all we have to
check is that $\pi^\al\Dmod_{\Bun(-1),\alpha}=\Dmod_{P,\beta}$.

By Proposition~\ref{BunP} we have $\Bun(-1)=P/\gm$, where $P$
is glued from two copies of $\p1$, which we denote now by
$\pp1+$ and $\pp1-$ (so that $x_i^-\in\pp1-$). We saw that $P$
can be viewed as the moduli space of triples
$(L,\eta,O_\p1(-2)\hookrightarrow L_\eta)$ (cf.
Remark~\ref{RemP}). We shall be using the notation from the
proof of Proposition~\ref{BunP}. Let $\rho:\tilde P\to P$ be
the projection. We assume that $\rho^{-1}(\pp1+)$ is given by
$p'\ne0$, while $\rho^{-1}(\pp1-)$ is given by $p\ne0$.

\begin{Lemma}\label{DistingPar}
For all $i$ there exists a unique $(L,\eta)\in\Bun(-1)$ such
that $\eta_{x_i}=(O_\p1)_{x_i}\subset L_{x_i}$ and under the
above description of $\Bun(-1)$ this point corresponds to
$x_i^-$.
\end{Lemma}
\begin{proof}
Consider the composition $O_\p1(-\divisor)\hookrightarrow
L(-\divisor)\hookrightarrow L_\eta$. Clearly,
$\eta_{x_i}=(O_\p1)_{x_i}$ if and only if this composition is
zero at $x_i$ . This happens if and only if the rank of
$\phi':O_\p1(-\divisor)\oplus O_\p1(-2)\rightarrow L_\eta$
drops at $x_i$ with the kernel $O_\p1(-\divisor)_{x_i}$. This
is in turn equivalent to $q=x_i$, $p'=0$.
\end{proof}

Let $\delta$ be the line bundle on $\Bun(-1)$ whose fiber at
$(L,\eta)$ is $\detrg(\p1,L)$. Let~$\delta'$ be the pullback of
$\delta$ to $P$. Fix $\infty\in\p1-\divisor$.

\begin{Lemma}
$\delta'\simeq O_P(2(\infty)-\sum n_i x_i^-)$.
\end{Lemma}
\begin{proof}
Let $t\in\gm$ act on $\tilde P$ by
\begin{equation}\label{action}
    t\cdot(p,p',q)=(p/t,tp',q).
\end{equation}
This action gives rise to a $\gm$-torsor $\tilde P\to P$. We
claim that the corresponding line bundle is $\delta'$.

Indeed, consider the cartesian diagram
\[
\begin{CD}
\tilde P @>\rho>> P\\
@VVV @VVV\\
P @>>> \Bun(-1).
\end{CD}
\]
Here the left hand arrow is the torsor described above. The top
arrow corresponds to forgetting the embedding $O_\p1(-2)\to
L_\eta$. Thus $P$ on the right parameterizes parabolic bundles
with embeddings $O_\p1\to L$. However, such an embedding is the
same as a non-zero element of $\detrg(\p1,L)=H^0(\p1,L)$. Hence
the torsor on the right is the one corresponding to $\delta$,
and the torsor on the left is the one corresponding to
$\delta'$.

We have
\[
 \rho^{-1}(\pp1+)=\{(f(q)/p',p',q)\in\tilde P\},
\]
where $p'\ne0$ is in the fiber of $O_\p1(2)$ over $q$. Thus
$\rho^{-1}(\pp1+)$ is the total space of $O_\p1(2)$ with the
zero section removed, and the action~\eqref{action} is the
standard one. Hence $\delta'|_{\pp1+}=O_\p1(2)$.

Further,
\[
 \rho^{-1}(\pp1-)=\{(p,f(q)/p,q)\in\tilde P\}
\]
is also the total space of $O_\p1(2)$ with the zero section
removed but the action~\eqref{action} is the inverse one, so
the total space of the corresponding line bundle is obtained by
compactifying at infinity and $\delta'|_{\pp1-}=O_\p1(-2)$. We
also see that if a meromorphic section $s$ of $\delta'$ has
order $m_i$ at $x_i^+$, then it has order $m_i-n_i$ at $x_i^-$.

Let $s$ be a section of $O_\p1(2)=\delta'|_{\pp1+}$ with a
double zero at $\infty$. We view it as a meromorphic section of
$\delta'$. It has no other zeroes on $\pp1+$, and, by the
previous remark, it has a pole of order $n_i$ at $x_i^-$. Thus
the divisor of $s$ is $2(\infty)-\sum n_i x_i^-$.
\end{proof}

Denote by $\Delta_-$ and $\Delta_+$ the graphs of the
immersions $\divisor\hookrightarrow\pp1-\hookrightarrow P$ and
$\divisor\hookrightarrow\pp1+\hookrightarrow P$, respectively.

\begin{Lemma}\label{Det}
Let $\cL$ be the universal family on $\p1\times P$. Then
\[
\det\cL|_{\divisor\times P}\simeq
O_{\divisor\times P}(\Delta_-+\Delta_+)\otimes p_2^*\delta'.
\]
\end{Lemma}
\begin{proof}
We have a canonical map $p_1^*O_\p1(-2)\to\cL$ (recall the
modular description of~$P$). On the other hand, we have an
adjunction morphism $p_2^*\delta'\to\cL$ (recall that the fiber
of $p_2^*\delta'$ is $H^0(\p1,L)$). These maps give rise to a
map $p_1^*O_\p1(-2)\otimes p_2^*\delta'\to\det\cL$, and it
vanishes exactly over the graph of $\wp$. Restricting to
$\divisor\times P$ we obtain the result.
\end{proof}

Note that $\gmd$-torsors on a scheme $Y$ are the same as
$\divisor$-families of line bundles on $Y$ (i.e., line bundles
on $\divisor\times Y$). Indeed, a line bundle on $\divisor\times Y$
is the same as a rank one locally free module over $\C[\divisor]\otimes_\C O_Y$.
Such modules are in one-to-one correspondence with torsors over the sheaf
$(\C[\divisor]\otimes_\C O_Y)^\times$.

Thus $\eta_{univ}$ and $\eta_{univ}'$ can be viewed as line
bundles on $\divisor\times\Bun(-1)$. Clearly,~$\eta_{univ}$ is
a subbundle of $\cL|_{\divisor\times\Bun(-1)}$, and
$\eta'_{univ}=(\cL|_{\divisor\times\Bun(-1)})/\eta_{univ}$.

\begin{Proposition}
\noindstep\label{Pullb1} The pullback of $\eta_{univ}$ to
$\divisor\times P$ is $O_{\divisor\times P}(\Delta_+)$.

\noindstep\label{Pullb2} The pullback of $\eta'_{univ}$ to
$\divisor\times P$ is $p_2^*(O_P(2(\infty)-\sum n_i
x_i^-))\otimes O_{\divisor\times P}(\Delta_-)$.
\end{Proposition}

\begin{Remark}
The asymmetry is due to the choice of one of two torsors
$P\to\Bun(-1)$.
\end{Remark}

\begin{proof}
Note that
$\pi^*\eta_{univ}\otimes\pi^*\eta'_{univ}=\det\cL|_{\divisor\times
P}$, thus~(\ref{Pullb1}) follows from~(\ref{Pullb2}) and
Lemma~\ref{Det}. Let us prove (\ref{Pullb2}). The proof is
essentially a family version of Lemma~\ref{DistingPar}.

As in the proof of Lemma~\ref{Det} we get a map
$\bar\delta:=p_2^*\delta'\to\cL$. Restricting this to
$\divisor\times P$ and composing with the natural projection we
get a map
\[
    \bar\delta|_{\divisor\times P}\to\pi^*\eta'_{univ}.
\]
We need to show that it vanishes exactly on $\Delta_-$.
Clearly, this map vanishes on $S\subset\divisor\times P$ if and
only if $\bar\delta\to\cL$ factors through $\eta_{univ}$ over
$S$. One checks that this happens if and only if
$\bar\delta(-(\divisor\times P))\to\cL(-(\divisor\times
P))\to\cL_{\eta_{univ}}$ vanishes over $S$. Let us show that
$S\subset\Delta_-$ in this case (we leave the converse to the
reader).

We see that the rank of $\phi':\bar\delta(-(\divisor\times
P))\oplus p_1^*O_\p1(-2)\rightarrow\cL_{\eta_{univ}}$ drops on
$S$. Recall the modular definition of $\wp$: its graph is given
by the scheme, where the rank of $\phi'$ drops (cf. proof of
Proposition~\ref{BunP}, Step~5). Thus
$S\subset\Delta_-\cup\Delta_+$. But the kernel of $\phi'$ is
$\bar\delta(-(\divisor\times P))$, thus in fact
$S\subset\Delta_-$.
\end{proof}

Now let us be explicit about what we need to calculate:
$\pi^*\eta_{univ}$ and $\pi^*\eta'_{univ}$ correspond to
classes $[\pi^*\eta_{univ}],[\pi^*\eta'_{univ}]\in
H^1(\divisor\times P, O^\times_{\divisor\times P})$. There is a
natural map $\mathbf{dlog}: O^\times_{\divisor\times P}\to
p_2^*\Omega_P: f\mapsto f^{-1}\dif_P f$. Applying this map to
$[\pi^*\eta_{univ}]$ and $[\pi^*\eta'_{univ}]$ we get elements
of $H^1(P,\Omega_P)\otimes_\C O_\divisor$. The TDO ring
$\pi^\al\Dmod_{\Bun(-1),\alpha}$ corresponds to an element of
$H^1(P,\Omega_P)$ given by
\[
\langle\mathbf{dlog}[\pi^*\eta_{univ}],(\alpha_i^+)\rangle+
\langle\mathbf{dlog}[\pi^*\eta'_{univ}],(\alpha_i^-)\rangle,
\]
where
\[
    \langle\cdot,\cdot\rangle:H^1(P,\Omega_P)\otimes
    O_\divisor\otimes O_\divisor^\vee\to H^1(P,\Omega_P).
\]
Choose local parameters $z_i$ at $x_i\in\p1$; we obtain an
isomorphism
\[
    O_\divisor=\prod_i\C[w_i]/w_i^{n_i}.
\]
Recall the
description of $H^1(P,\Omega_P)$ given in Lemma~\ref{CohP}. An
easy calculation shows that
\[
\begin{split}
\mathbf{dlog}\left(p_2^*(2(\infty)-\sum n_i x_i^-)\right)&=
(0,n_i\,\dif z_i/z_i)\otimes1_\divisor,\\
\mathbf{dlog}(O_{\divisor\times P}(\Delta_-))&=
\left(1_\divisor,-\frac{\dif z_i}{z_i-w_i}
\right),\\
\mathbf{dlog}(O_{\divisor\times P}(\Delta_+))&=
\left(1_\divisor,\frac{\dif z_i}{z_i-w_i}
\right)
\end{split}
\]
($\frac{\dif z_i}{z_i-w_i}$ should be expanded in the powers of
$w_i$). Further,
\[
\begin{split}
\langle(0,n_i\,\dif z_i/z_i)\otimes1_\divisor,(\alpha_i^-)\rangle&=
(0,n_i\lambda\,\dif z_i/z_i),\\
\left\langle\left(1_\divisor,-\frac{\dif z_i}{z_i-w_i}
\right),(\alpha_i^-)\right\rangle&=(\lambda,-\alpha_i^-),\\
\left\langle\left(1_\divisor,\frac{\dif z_i}{z_i-w_i}
\right),(\alpha_i^+)\right\rangle&=\left(\sum_i\res\alpha_i^+,\alpha_i^+\right).
\end{split}
\]
Note that collections $(\alpha_i^\pm)$ in the left-hand side
are viewed as elements of $O_\divisor^\vee$, while in the
right-hand side they are polar parts of 1-forms.

Applying the previous proposition and recalling that
$\lambda+\sum_i\res\alpha_i^+=-d$, we see that the element of
$H^1(P,\Omega_P)$ corresponding to
$\pi^\al\Dmod_{\Bun(-1),\alpha}$ is
\[
(-d,\alpha_i^+-\alpha_i^-+n_i\lambda\,\dif z_i/z_i).
\]
It remains to notice that $\beta_i^\pm$ correspond to the same
element of $H^1(P,\Omega_P)$, cf. Lemma~\ref{CohP}.

\subsection{Step 3: $\cM_\alpha\simeq\cM_\beta$.}

This isomorphism is provided by Katz's middle convolution. It
is defined in~\cite{Katz} in the settings of $l$-adic sheaves;
see~\cite{SimpsonMiddleConvolution} or~\cite{ArinkinFourier} for the
settings of de Rham local systems. Here is an explicit
description of the isomorphism.

Fix $\infty\in\p1-\divisor$. There is a unique 1-form $\alpha$
on $\p1-\divisor-\{\infty\}$ such that $\alpha+\alpha_i^-$ is
non-singular at $x_i$ and $\alpha$ has a pole of order one at
$\infty$.  Similarly, there is a unique 1-form $\beta$ on
$\p1-\divisor-\{\infty\}$ such that $\beta+\beta_i^-$ is
non-singular at $x_i$ and $\beta$ has a pole of order one at
$\infty$. Note that
$\res_\infty\alpha=\sum_i\res\alpha_i^-=\lambda$ and
$\res_\infty\beta=-\lambda$.

Fix $(L,\nabla)\in\cM_\alpha$. The connection
\[\nabla+\alpha:L\to L\otimes\Omega_\p1(\divisor+(\infty))\]
has formal type $(0,\alpha_i^+-\alpha_i^-)$ at $x_i$. Let
$\tilde L\subset L$ be the largest subsheaf such that
\[(\nabla+\alpha)(\tilde L)\subset L\otimes\Omega_\p1(\infty).\]
Explicitly, $\tilde L$ is the modification of $L$ with respect
to one of two parabolic structures on $L$ induced by $\nabla$.
Precisely, this parabolic structure $\eta$ is such that the
polar part of $\nabla$ induces multiplication by $\alpha_-$ on
$\eta$ (cf. Corollary~\ref{ConnExist}).

Consider on $\p1\times\p1$ the differential 1-form
$\lambda\dif\log(x-y)$, where $x$ and $y$ are the coordinates
on the first and second factors, respectively. The preimage
$p_1^*L$ carries a flat meromorphic connection $p_1^*\nabla$;
let us equip $p_1^*L$ with the flat meromorphic connection
\[
    p_1^*\nabla+p_1^*\alpha+\lambda\dif\log(x-y).
\]
Denote the `horizontal' and `vertical' parts of this connection
by $\nabla_x$ and $\nabla_y$. We then obtain an
anti-commutative square
{\small\begin{equation}\label{twonablas}
\begin{CD}
p_1^*\tilde L @>\nabla_x>> p_1^*L\otimes p_1^*\Omega_\p1(\Delta)\\
@V\nabla_yVV @V\nabla_yVV\\
p_1^*\tilde L\otimes p_2^*\Omega_\p1(\Delta+\p1\times\{\infty\})@>\nabla_x>>
p_1^*L\otimes p_1^*\Omega_\p1\otimes p_2^*\Omega_\p1(2 \Delta+\p1\times\{\infty\}).
\end{CD}
\end{equation}}

Consider the complex
\[\begin{CD}
\cF^\al:=(p_1^*\tilde L @>\nabla_x>> p_1^*L\otimes p_1^*\Omega_\p1(\Delta))
\end{CD}\]
of sheaves on $\p1\times\p1$. The differential $\nabla_x$ is
$p_2^{-1}O_\p1$-linear, so the direct image $Rp_{2,*}\cF^\al$
makes sense as an object in the derived category of
$O_\p1$-modules. It is easy to see that
$R^0p_{2,*}\cF^\al=R^2p_{2,*}\cF^\al=0$. Now the Euler
characteristic argument shows that $Rp_{2,*}\cF^\al[1]$ is a
locally free $O_\p1$-module of rank two; let us denote it
by~$E$.

Similarly, consider the complex
\[\begin{CD}
\cF^\al(\Delta)=(p_1^*\tilde L(\Delta) @>\nabla_x>>
p_1^*L\otimes p_1^*\Omega_\p1(2\Delta)).
\end{CD}\]
Then $Rp_{2,*}\cF^\al(\Delta)[1]$ is a locally free
$O_\p1$-module of rank two; let us denote it by~$\tilde E$.

The natural morphism $\cF^\al\hookrightarrow\cF^\al(\Delta)$
induces a homomorphism $\iota:E\to\tilde E$. Recall that
\[
    O_{\p1\times\p1}(k\Delta)/O_{\p1\times\p1}((k-1)\Delta)
    \approx(\iota_\Delta)_*\cT_\p1^{\otimes k}.
\]
Thus we have an exact sequence of complexes
\[
    0\to\cF^\al\to\cF^\al(\Delta)\to
    (\iota_\Delta)_*(\tilde L\otimes\cT_\p1\to L\otimes\cT_\p1)\to0.
\]
One checks that the differential in the rightmost complex is
induced by the natural inclusion $\tilde L\hookrightarrow L$.
Thus $\iota$ is an embedding, and $\coker(\iota)\simeq
O_\divisor$. We can thus identify $\tilde E$ with an upper
modification of $E$.

Finally, note that diagram~\eqref{twonablas} provides a
$\C$-linear map $\nabla_E:E\to\tilde
E\otimes\Omega_\p1(\infty)$. Clearly, $\nabla_E$ satisfies the
Leibnitz identity. We view $\nabla_E$ as a connection on $E$
with poles at $\divisor\cup\{\infty\}$.

\begin{Proposition}
The formal type of $\nabla_E$ at $x_i$ is
$(0,\beta_i^+-\beta_i^-)$, and the residue of~$\nabla_E$ at
$\infty$ is $-\lambda$. In other words, $(E,\nabla_E-\beta)\in
\cM_\beta$. The correspondence
\[(L,\nabla)\mapsto(E,\nabla_E-\beta)\]
is an isomorphism $\cM_\alpha\iso\cM_\beta$.
\end{Proposition}
We shall prove a slightly weaker statement, which is sufficient
for our purposes. Namely, we prove that $(E,\nabla_E)$
has the described formal types after a modification.
\begin{proof}
Let $\Dmod_{\p1,\lambda}$ be the TDO ring corresponding to
$\lambda\in\C=H^1(\p1,\Omega_\p1)$. For
$(L,\nabla)\in\cM_\alpha$ consider the
$\Dmod_{\p1,\lambda}$-module $\jmath_{!*}(L,\nabla+\alpha)$,
where $\jmath:\p1-(\divisor\cup\infty)\hookrightarrow\p1$ is
the natural inclusion. Note that it has no singularity at
$\infty$ because of the twist by $\lambda$.

As explained in~\cite[Section~6.3]{ArinkinFourier}, the Katz--Radon
transform gives an equivalence
$\fR:\Dmod_{\p1,\lambda}-\mathrm{mod}\to\Dmod_{\p1,-\lambda}-\mathrm{mod}$
and it is easy to see that it is compatible with our
construction in the sense that
\[
    \fR(\jmath_{!*}\jmath^*(L,\nabla+\alpha))=
    \jmath_{!*}\jmath^*(E,\nabla_E).
\]
(The proof is similar to~\cite[Lemma~13]{Arinkin}.)

Let $\Phi_{x_i}$ be the functor of vanishing cycles as defined
in~\cite{ArinkinFourier}. We get
\begin{multline*}
\Phi_{x_i}(\jmath_{!*}\jmath^*(E,\nabla_E))=
\Phi_{x_i}\fR(\jmath_{!*}\jmath^*(L,\nabla+\alpha))=
\fR(x_i,x_i)\Phi_{x_i}\jmath_{!*}\jmath^*(L,\nabla+\alpha)=\\
\fR(x_i,x_i)(O_{\dot D},\dif+\alpha_i^+-\alpha_i^-)= (O_{\dot
D},\dif+\alpha_i^+-\alpha_i^-+n_i\lambda).
\end{multline*}
Here $\fR(x_i,x_i)$ is the local Katz--Radon transform, the
second equality is~\cite[Corollary~6.11]{ArinkinFourier}, the
last equality is~\cite[Theorem~C]{ArinkinFourier}. Now it is
easy to see that $\jmath_{!*}\jmath^*(E,\nabla_E)$ has required
singularities.

It follows that the formal type of $\nabla_E$ at $x_i$ is
$(m_i,\beta_i^+-\beta_i^-+m_i')$, where $m_i$, and~$m_i'$ are
integers. Thus $(E,\nabla_E-\beta)$ becomes a connection in
$\cM_\beta$ after a suitable modification.

Note that the construction of $(E,\nabla_E-\beta)$ works in
families as well. Since formal normal forms of connections
exist in families, after a suitable modification we get a
morphism $\cM_\alpha\to\cM_\beta$. If we do the same
construction with $\alpha$ and $\beta$ switched and use the
inverse Katz--Radon transform, we get a morphism
$\cM_\beta\to\cM_\alpha$. It is easy to see that
these morphisms are inverse to each
other.
\end{proof}
This completes the proof of Theorem~\ref{Langlands}.

\bibliography{IrregularLanglands3}
\end{document}